\documentclass[12pt]{article}

\usepackage{graphicx}
\usepackage{latexsym,amssymb}
\usepackage{amsthm}
\usepackage{indentfirst}
\usepackage{amsmath}
\usepackage{color}
\usepackage{xcolor}
\usepackage{fourier}
\usepackage[colorlinks=true,backref=page]{hyperref}
\usepackage[all]{xy}

\usepackage[a4paper,text={160true mm,230true mm},centering,top=35true mm,head=5true mm,headsep=2.5true mm,foot=8.5true mm]{geometry}

\newtheorem{theoremalph}{Theorem}

\newtheorem*{Main Theorem}{Main Theorem}

\newtheorem{Theorem}{Theorem}[section]
\newtheorem*{Theorem A}{Theorem A}
\newtheorem*{Theorem A'}{Theorem A'}
\newtheorem*{Theorem B'}{Theorem B'}

\newtheorem*{Conjecture}{Conjecture}

\newtheorem{Proposition}[Theorem]{Proposition}
\newtheorem{Lemma}[Theorem]{Lemma}

\newtheorem{Remark}[Theorem]{Remark}

\newtheorem{Remark-numbered}[Theorem]{Remark}
\newtheorem{Corollary}[Theorem]{Corollary}

\newtheorem{Claim-numbered}[Theorem]{Claim}

\def\HH{{\mathbb H}}

 \def\NN{{\mathbb N}} 
\def\PP{{\mathbb P}}

  \def\cH{{\cal H}}  
\def\cC{{\cal C}}    
   \def\cP{{\cal P}} 
\def\cE{{\cal E}}   \def\cQ{{\cal Q}} 
\def\cF{{\cal F}}  \def\cL{{\cal L}} \def\cR{{\cal R}}

\def\Id{\operatorname{Id}}

\def\dim{\operatorname{dim}}

\def\diam{\operatorname{Diam}}

\begin{document}

\title{Symbolic extension of $\mathcal{C}^{1,\alpha}$ maps in arbitrary dimensions}

\author
{Chiyi Luo and Dawei Yang\footnote{
	D. Yang  was partially supported by National Key R\&D Program of China (2022YFA1005801), NSFC (12325106 \& 12526207), ZXL2024386 and Jiangsu Specially Appointed Professorship.
	C. Luo was partially supported by NSFC (12501244) and Early-Career Young Scientists and Technologists Project of Jiangxi Province. }
}

\date{}
\maketitle

\begin{abstract} 
We prove that every $\mathcal{C}^{1,\alpha}$ self-map of a compact Riemannian manifold, in arbitrary dimension, admits a symbolic extension.
This gives a positive answer to the conjecture of Downarowicz and Newhouse in arbitrary dimensions.
Our method is based on analyzing the small singular values of the linearized orbit operator.
\end{abstract}

\tableofcontents	

\section{Introduction}\label{SEC:1}
A symbolic extension of a topological dynamical system $(M,f)$ is a subshift $(\Sigma,\sigma)$ over a finite alphabet $\mathcal A$ together with a continuous surjective map $\pi:\Sigma\to M$ satisfying $\pi\circ\sigma=f\circ\pi$.
Here, $\Sigma\subset\mathcal A^{\mathbb Z}$ is a closed subset satisfying $\sigma(\Sigma)\subset\Sigma$ and $\sigma$ denotes the left shift map.

Finite-alphabet subshifts are a well-understood class of dynamical systems with a concrete combinatorial description. 
If a topological dynamical system admits a symbolic extension, then it is a topological factor of a finite-alphabet subshift. 
Thus, its dynamics can be represented by a finite-alphabet symbolic model.
Moreover, topological entropy does not increase under factor maps, so the entropy of the original system is bounded above by that of the finite-alphabet subshift and is therefore finite.

However, finite topological entropy alone does not guarantee the existence of a symbolic extension.
In 1990, Boyle constructed a zero dimensional finite entropy counterexample with no symbolic extension; see \cite[Example 3.1]{BFF02} for a record of this example.
This naturally raises the question of whether differentiability rules out such obstructions.
The answer remains negative in the $\mathcal C^1$ category: although every $\mathcal C^1$ self-map of a compact Riemannian manifold has finite topological entropy, Downarowicz-Newhouse \cite{DN05} proved that generic non-Anosov $\mathcal C^1$ area-preserving surface diffeomorphisms admit no symbolic extensions.

This leads to two fundamental questions: which systems admit symbolic extensions, and, when one exists, what is the smallest possible gap between the entropy of the symbolic extension and that of the original system?
A symbolic extension with zero entropy gap is called principal.
It preserves the entropy of every invariant measure: each invariant measure on the symbolic system has the same entropy as its projection onto the original system. Consequently, a principal symbolic extension adds no entropy and has the same topological entropy as the original system.

By contrast, the situation is much better in the $\cC^\infty$ category. 
Based on the Gromov-Yomdin theory \cite{Gro87,Yom87}, Buzzi \cite{Buz97} proved that every $\cC^\infty$ self-map of a compact Riemannian manifold is asymptotically $h$-expansive. 
Combined with the symbolic extension theory for asymptotically $h$-expansive systems (see Boyle-Fiebig-Fiebig \cite{BFF02} and Downarowicz \cite{Dow05}), this implies that every such map admits a principal symbolic extension.
The sharp contrast between the $\cC^1$ and $\cC^\infty$ categories raises a natural question: how much regularity beyond $\cC^1$ is sufficient to guarantee the existence of a symbolic extension? 
Motivated by this question, Downarowicz-Newhouse formulated the following conjecture.

\begin{Conjecture}[Conjecture 1.1, \cite{DN05}]
	Every $\cC^r$ self-map with $r\geq 2$ of a compact Riemannian manifold has a symbolic extension.
\end{Conjecture}

Since then, a stronger $\cC^{1,\alpha}$ version of this conjecture has been established for interval and circle maps by Downarowicz-Maass \cite{DM09}, for maps on surfaces by Burguet \cite{Bur11,Bur12}, and for three-dimensional diffeomorphisms by Burguet-Liao \cite{BL2022}.

Although these results guarantee the existence of symbolic extensions, such extensions cannot in general be chosen to be principal at finite smoothness. Indeed, Downarowicz-Newhouse showed that, for every finite $r\geq2$, there are open sets of $\cC^r$ surface diffeomorphisms in which a residual subset has positive residual entropy and hence admits no principal symbolic extension \cite{DN05}. 
On the other hand, additional hyperbolic structure can ensure the existence of principal symbolic extensions. 
For instance, uniformly hyperbolic systems admit principal symbolic codings via Markov partitions \cite{Bow75}.

Let $f: M \to M$ be a continuous map on a compact metric space $M$ with finite topological entropy.
Denote by $h_{\rm top}(\cdot)$ the topological entropy. 
Define the symbolic extension entropy  
$$h_{\rm sex}(f):=\inf\big\{h_{\rm top}(\sigma):(\Sigma,\sigma)~\text{is a symbolic extension of}~(f,M)\big\},$$
with the convention that $h_{\rm sex}(f)=\infty$ if no such symbolic extension exists. 

Downarowicz-Newhouse formulated the following quantitative problem concerning an estimate of $h_{\rm sex}(f)$, motivated by the estimates of Gromov \cite{Gro87} and Yomdin \cite{Yom87}.
\begin{Conjecture}[Conjecture 1.2, \cite{DN05}]
	Every $\cC^r$ self-map $f:M\to M$ with $r\geq 2$ of a compact Riemannian manifold $M$ has a symbolic extension, and 
	$$h_{\rm sex}(f) \leq \frac{rd}{r-1} \cdot \lim_{n\to \infty}\frac{1}{n}\log^{+}\big(\sup_{x\in M}\|D_xf^n\|\big).$$
\end{Conjecture}
This conjecture was established for interval and circle maps by Downarowicz-Maass \cite{DM09}, for surface maps by Burguet \cite{Bur11,Bur12}, and for three-dimensional diffeomorphisms by Burguet-Liao \cite{BL2022}.
In dimension one, Burguet \cite{Bur10} constructed examples showing that this bound is sharp, and for surface diffeomorphisms, he \cite{Bur11} also obtained a sharp estimate for the entropy gap.
It is worth noting that all previously known results on the existence of symbolic extensions and entropy boundedness for general finite-regularity maps, without additional hyperbolic assumptions, are restricted to manifolds of dimension at most three.
This restriction arises from existing reparametrization arguments, which reduce local entropy to the growth of curves in dimensions two and three, but fail to control higher-dimensional disks in the absence of the $\cC^\infty$ assumption.

We now state our main result. 
We establish the existence of symbolic extensions and provide an entropy bound for $\cC^{1,\alpha}$ self-maps of compact Riemannian manifolds in arbitrary dimensions.

\begin{theoremalph}\label{Thm:Main-SE}
	Let $M$ be a $d$-dimensional compact Riemannian manifold. 
	Then, every $\mathcal C^{1,\alpha}$ map $f:M\to M$ with $\alpha\in (0,1]$ admits a symbolic extension, and  
	$$h_{\rm sex}(f)\leq  h_{\rm top}(f)+ \frac{d}{\alpha}\cdot R(f),~~\text{where}~~R(f):=\lim_{n\to \infty}\frac{1}{n}\log^{+}\big(\sup_{x\in M}\|D_xf^n\|\big).$$
	Moreover, there is a symbolic extension $\pi:(\Sigma,\sigma)\to(M,f)$ such that for any $f$-invariant measure $\mu$,
	$$\max\big\{h_{\nu}(\sigma): \nu~\text{is}~\sigma\text{-invariant and}~\pi_{\ast}(\nu)=\mu \big\}=h_{\mu}(f)+\frac{1}{\alpha} \cdot \lambda_{\Sigma}^+(\mu,f).$$
\end{theoremalph}

Here, $\lambda_{\Sigma}^+(\mu,f)$ denotes the integral of the sum of the positive Lyapunov exponents associated with the $f$-invariant measure $\mu$ (see Section \ref{SEC:Entropy}), and $h_\mu(f)$ denotes the metric entropy of $\mu$ with respect to $f$.

Theorem \ref{Thm:Main-SE} provides an affirmative answer to \cite[Conjecture 1.1]{DN05} and strengthens it by showing that $\cC^{1,\alpha}$ regularity already suffices for the existence of symbolic extensions.
Moreover, by the Ruelle inequality \cite{Rue78} and the variational principle \cite[Theorem 8.6]{Wal82}, the topological entropy of every $\cC^1$ map $f$ is bounded above by $dR(f)$.
Therefore, Theorem \ref{Thm:Main-SE} confirms the $\mathcal C^{2}$ version of \cite[Conjecture 1.2]{DN05} and, in fact, extends it to $\cC^{1,\alpha}$ maps.

The proof of Theorem \ref{Thm:Main-SE} relies on the symbolic extension theorem (see the version stated for continuous maps in \cite[Theorem 3]{Bur12}; see also \cite[Theorem 5.5]{BD04} and \cite[Chapter 9]{Dow11}) and the uniform estimate between metric entropy and partition entropy (see our Theorem \ref{Thm:Main-Entropy}).

We first recall the symbolic extension theorem.
For convenience, denote by $\mathbb{P}_{f}(M)$ the set of all $f$-invariant probability measures.
The symbolic extension theorem is formulated in terms of entropy structures (see \cite{Dow05} and Appendix \ref{SEC:A}).
For our purpose, we fix the entropy structure defined as follows.
Let $\{\varepsilon_n\}_{n\geq1}$ be a decreasing sequence of positive numbers converging to zero, and define
$$ h_n(\mu):=\inf\left\{h_\mu(f,\mathcal Q):\diam(\mathcal Q)<\varepsilon_n\right\}, \quad \mu\in\mathbb P_f(M),$$
where the infimum is taken over all finite Borel partitions $\mathcal Q$ of $M$.
Then $\{h_n\}_{n\geq1}$ is an entropy structure, since it is uniformly equivalent to the continuous-function entropy structure (see Appendix \ref{SEC:A}). 
We state the theorem with respect to this entropy structure; the original version is stated in Theorem \ref{Thm:A.1}.

\begin{Theorem}[Symbolic extension theorem] \label{Thm:1.1}
	Let $f:M\to M$ be a continuous map with finite topological entropy.
	Assume that $\lambda(\cdot)$ is a non-negative affine upper semi-continuous function defined on $\PP_{f}(M)$, and that $\{\varepsilon_n\}_{n\ge1}$ is a decreasing sequence converging to zero. 
	If for every $\gamma>0$ and every $\mu\in \PP_{f}(M)$, there exists $n_{\mu}>0$ for which
	$$\limsup_{\nu \to \mu} \Big(h_{\nu}(f)- \inf \big\{h_{\nu}(f,\cQ):\diam (\cQ)<\varepsilon_n\big\}+\lambda(\nu)\Big)\leq \lambda(\mu)+\gamma,~\forall n\geq n_{\mu}.$$
	Then, there exists a symbolic extension $\pi:(\Sigma,\sigma)\to (M,f)$ such that 
	$$\max\big\{h_{\nu}(\sigma): \nu~\text{is}~\sigma\text{-invariant and}~\pi_{\ast}(\nu)=\mu \big\}=h_{\mu}(f)+\lambda(\mu).$$
	In particular, $h_{\rm sex}(f)\leq h_{\rm top}(f)+\sup\big\{\lambda(\mu):\mu\in \PP_{f}(M)\big\}$.
\end{Theorem}
Next, we state Theorem \ref{Thm:Main-Entropy}, which establishes a uniform approximation of metric entropy by partition entropy using uniform partitions for $\mathcal C^{1,\alpha}$ maps $f:M\to M$, $\alpha\in(0,1]$.
The notation $\|f\|_{\cC^{1,\alpha}}$ will be introduced in Section \ref{SEC:Entropy}, and $\wedge^k$ denotes the $k$-th exterior power operator.
\begin{theoremalph}\label{Thm:Main-Entropy}
	Let $M$ be a $d$-dimensional compact Riemannian manifold. 
	For each $\alpha\in (0,1]$, $\Upsilon>1$ and $q\in\mathbb N$, there exist $\varepsilon_q=\varepsilon_{q}(\alpha,\Upsilon)>0$ and $\delta_q=\delta_{q}(\alpha,d,\Upsilon)>0$ such that 
	\begin{itemize}
		\item  $\delta_{q}(\alpha,d,\Upsilon)\to 0$ as $q\to \infty$;
		\item  for every $\mathcal{C}^{1,\alpha}$ map $f:~M\to M$ satisfying $\|f\|_{\mathcal C^{1,\alpha}}\leq \Upsilon$;
		\item  for every invariant measure $\mu$ of $f$ and every finite partition $\mathcal Q$ with ${\rm Diam}(\mathcal Q)<\varepsilon_q$;
	\end{itemize}
	one has 
	$$h_\mu(f)\le h_{\mu}(f,\cQ)+\frac{1}{\alpha}\Big(\frac{1}{q}\int \max_{1\leq k\leq d}~\log^{+} \|\wedge^k D_xf^q\|~{\rm d}\mu(x)-\lambda^{+}_{\Sigma}(\mu,f)\Big)+\delta_{q}.$$
\end{theoremalph}
Theorem \ref{Thm:Main-Entropy} provides the key ingredient for Theorem \ref{Thm:1.1} and shows that the function $\lambda(\mu)$ can be chosen as $\lambda_{\Sigma}^+(\mu,f)/\alpha$.
This result was previously stated for diffeomorphisms in dimensions at most three in \cite[Theorem B]{LuY25}, however, it also follows from the previous works \cite{Bur12,BL2022}.
The proof in dimensions two and three is based on Burguet's reparametrization lemma \cite{Bur12}, which extends the Yomdin-Gromov theory \cite{Gro87,Yom87} to one-dimensional curves and provides a remarkable upper bound.
Extending Burguet's refinement of the Yomdin-Gromov theory from one-dimensional curves to arbitrary dimensions remains a difficult problem. 
We show that this result remains valid in arbitrary dimensions without assuming that $f$ is a diffeomorphism.
Our approach is independent of Yomdin's theory, Pesin theory, and related techniques.
An outline of our approach will be given at the end of this section.

A deep result of Newhouse \cite{New89}, based on Yomdin's estimate \cite{Yom87}, shows that there exists $C_r>0$ such that, for every $\cC^{r}$ map $f:M\to M$ and every $q\in\NN$, there exists $\varepsilon_q>0$ satisfying
$$h_\mu(f)\leq h_\mu(f,\cQ)+ \frac{\dim M\cdot \log~\sup_{x\in M}\|D_xf^q\|}{qr}+\frac{\log(C_r)}{q},$$
for every partition $\cQ$ with $\diam(\cQ)<\varepsilon_q$.
Notice that, in the $\cC^{1,\alpha}$ setting, our result provides a better bound than Newhouse's estimate.
We now prove Theorem \ref{Thm:Main-SE} using Theorem \ref{Thm:Main-Entropy} and the symbolic extension theorem \ref{Thm:1.1}.

\begin{proof}[Proof of Theorem \ref{Thm:Main-SE}]
	Let $f:M\to M$ be a $\mathcal{C}^{1,\alpha}$ map with $\|f\|_{\cC^{1,\alpha}}\leq \Upsilon$.
	Let $\{r_q\}_{q>0}$ be a decreasing sequence converging to zero such that $0<r_q\leq \varepsilon_q(\alpha,d,\Upsilon)$ for each $q>0$, where $\varepsilon_q(\alpha,d,\Upsilon)$ is the constant appearing in Theorem \ref{Thm:Main-Entropy}.
	For every $\gamma>0$ and every $f$-invariant measure $\mu$, choose $q_{\mu}>0$ such that for every $q\geq q_{\mu}$ ($\delta_q$ is as in Theorem \ref{Thm:Main-Entropy}), 
	$$ \frac{1}{\alpha}\Big(\frac{1}{q}\int \max_{1\leq k\leq d}~\log^{+} \|\wedge^k D_xf^q\|~{\rm d}\mu(x)-\lambda^{+}_{\Sigma}(\mu,f)\Big)+\delta_q<\gamma. $$
	Then, by Theorem \ref{Thm:Main-Entropy}, for every $q\geq q_{\mu}$ we have 
	\begin{align*}
		     &\limsup_{\nu \to \mu}~\Big[\sup \big\{h_{\nu}(f)-h_{\nu}(f,\cQ):\diam(\cQ)<r_q\big\}+\frac{1}{\alpha} \cdot \lambda^{+}_{\Sigma}(\nu,f) \Big]         \\
		\leq~& \limsup_{\nu \to \mu}~\Big[ \delta_q + \frac{1}{\alpha} \Big( \frac{1}{q} \int \max_{1\leq k\leq d}~\log^{+} \|\wedge^k D_xf^q\|~{\rm d}\nu(x) \Big) \Big] \\ 
		\leq~& \delta_q+ \frac{1}{\alpha}\Big( \frac{1}{q} \int \max_{1\leq k\leq d}~\log^{+} \|\wedge^k D_xf^q\|~{\rm d}\mu(x) \Big) \leq~\frac{1}{\alpha} \cdot \lambda^{+}_{\Sigma}(\mu,f)+\gamma.
	\end{align*}
	Hence, we can take $\lambda(\cdot)=\lambda^{+}_{\Sigma}(\cdot,f)/\alpha$, which is a non-negative affine upper semi-continuous function defined on the space of all $f$-invariant measures.
	Thus, Theorem \ref{Thm:Main-SE} follows immediately from Theorem \ref{Thm:1.1}.
\end{proof}
Newhouse \cite{New89} proved that the entropy map of $\cC^\infty$ systems is upper semi-continuous, while Misiurewicz \cite{Mis73} and Buzzi \cite{Buz14} provided examples showing that the entropy map of $\mathcal C^r$ systems need not be upper semi-continuous.
As an application of Theorem \ref{Thm:Main-Entropy}, we further show that the defect of upper semi-continuity of the entropy map is bounded by the sum of the positive Lyapunov exponents. 
This provides a positive answer to the $\cC^{1,\alpha}$ version of Burguet's conjecture \cite[Conjecture 1]{Bur11} .

\begin{theoremalph}\label{Thm:Main-USC}
	Let $M$ be a $d$-dimensional compact Riemannian manifold and $\alpha\in(0,1]$.
	Let $f:M\to M$ be a $\mathcal C^{1,\alpha}$ map with an invariant probability measure $\mu$.
	Consider sequences of $\mathcal C^{1,\alpha}$ maps $\{f_n\}$ and probability measures $\{\mu_n\}$ such that $f_n\to f$ in the $\mathcal C^{1,\alpha}$ topology and $\mu_n\to\mu$ in the weak$^\ast$ topology, with each $\mu_n$ being $f_n$-invariant.
	Then, we have 
	$$\limsup_{n\to\infty}h_{\mu_n}(f_n)-h_{\mu}(f)\leq \limsup\limits_{n\to\infty} \frac{1}{\alpha}\big(\lambda_{\Sigma}^+(\mu,f)-\lambda_{\Sigma}^+(\mu_n,f_n)\big).$$
	In particular, if $~\lambda_{\Sigma}^+(\mu_n,f_n) \to \lambda_{\Sigma}^+(\mu,f)$, then $\limsup\limits_{n\to\infty}h_{\mu_n}(f_n)\leq h_{\mu}(f)$.
\end{theoremalph}

Finally, we briefly describe the strategy of our proof of Theorem \ref{Thm:Main-Entropy}.
The key estimate concerns the number of arbitrarily small Bowen balls needed to cover a Bowen ball of a fixed size.
The main idea is based on a simple observation, also used in \cite{LMZ26}: if $\big(E,(\cdot,\cdot)\big)$ is a $d$-dimensional Euclidean inner product space, then any ball of radius $r$ can be covered by at most $(1+2r/\delta)^d$ balls of radius $\delta$; see also Lemma \ref{Lem:Acover}.
A Bowen ball $B(x,n,\varepsilon)$ (see \eqref{eq:dofBowen} for the definition) can be lifted to a subset of the $(n+1)d$-dimensional linear space
$V_n:=T_xM\oplus\cdots\oplus T_{f^n(x)}M$.
For each $v=(v_0,\ldots,v_n)\in V_n$, the Bowen norm $\max_{0\leq k\leq n}\|v_k\|_{f^k(x)}$
is not necessarily induced by an inner product.
However, the inner product $(\cdot,\cdot)_V:=\frac{1}{n+1}\sum_{k=0}^{n}(\cdot,\cdot)_{f^k(x)}$ induces a ``mean'' Bowen metric on $V_n$, and this metric is comparable with the Bowen metric (see Lemma \ref{Lem:key1}).
Hence, the use of the mean Bowen metric does not affect the entropy estimates, as in \cite{LuZ21}.
Under this inner product, we can study the covering of a Bowen ball.

Although $V_n$ has dimension $(n+1)d$, a mean Bowen ball is determined by the orbit of a point in $T_xM$, and its ``effective dimension'' is much smaller than $nd$.
In the extreme case where all the lifted maps $F^k:T_xM\to T_{f^k(x)}M$ are linear, 
$(u,v)_n:=\frac{1}{n+1}\sum_{k=0}^{n}(F^k(u),F^k(v))_{f^k(x)}$ defines an inner product on $T_xM$.
Therefore, the effective dimension of a mean Bowen ball is only $d$, independent of $n$.
However, in general, $F^j$ need not be linear. 
Hence, we define the \textit{linearized orbit operator} $\cL_x(v_0,\ldots,v_n)=\big(v_1-D_{x}f(v_0),~\cdots,v_n-D_{f^{n-1}x}f(v_{n-1})\big)$  (see Section \ref{Sec:CofX} for details), which compares the deviation between the nonlinear orbit and its linear approximation.
We show that the ``effective dimension'' of a Bowen ball is controlled by the number of small singular values of the corresponding linearized orbit operator (see Lemma \ref{Lem:Key2}).
Estimating this number yields the right-hand side of the inequality in Theorem \ref{Thm:Main-Entropy}.

The structure of this paper is as follows. 
In Section \ref{SEC:Entropy} and Section \ref{SEC:3}, we recall some preliminaries on entropy and linear algebra. 
In Section \ref{SEC:4}, we establish a key Theorem \ref{Thm:key}, which gives an estimate of the number of small Bowen balls needed to cover a large Bowen ball. 
The proofs of Theorem \ref{Thm:Main-Entropy} and Theorem \ref{Thm:Main-USC} are provided in the last two sections.

\section{Preliminaries on entropy and Lyapunov exponents}\label{SEC:Entropy}
Assume that $M$ is a $d$-dimensional compact boundaryless Riemannian manifold.
For a $\mathcal C^{1,\alpha}$ map $f:M\to M$ with $\alpha\in (0,1]$, we mean
\begin{itemize}
	\item if $\alpha=1$, it is a $\mathcal C^1$ map, and its derivative $Df$ is a Lipschitz map;
	\item if $\alpha\in(0,1)$, it is a $\mathcal  C^1$ map, and its derivative $Df$ is a $\alpha$-H\"older map.
\end{itemize} 
Assume that $X$ is a compact metric space and $Y$ is a Banach space.
For $\alpha\in (0,1]$ and a $\alpha$-H\"{o}lder continuous map $H:X\to Y$, define
\begin{equation}\label{eq:normd}
	\|H\|_0=\sup_{x\in X}\|H(x)\|,~~\|H\|_\alpha=\sup \left \{ \dfrac{d(H(x),H(y))}{d(x,y)^\alpha}:x\neq y,~x,y\in X \right\}.
\end{equation}
For $\mathcal C^{1,\alpha}$ map $f:~M\to M$, define
$$\|f\|_{\mathcal C^{1,\alpha}}:=\max\big\{\|Df\|_0,~\|Df\|_\alpha\big\}.$$

We recall some fundamental definition and properties of entropies.
Let $\mu$ be a  probability measure.
Given a finite partition $\mathcal P$, define the static entropy function
$$H_\mu(\mathcal P):=\sum_{P\in\mathcal P}-\mu(P)\log\mu(P).$$
By definition, one has 
\begin{equation}\label{eq:estimate-conditional}
	H_\mu(\mathcal P)\le \log\#\{P\in\mathcal P:~\mu(P)>0\}.
\end{equation}
Given two finite partitions $\mathcal P$ and $\mathcal Q=\{Q_1,\cdots,Q_k\}$, define the conditional entropy
$$H_\mu(\mathcal P\big|\mathcal Q):=\sum_{j=1}^k\mu(Q_j)H_{\mu_j}(\mathcal P),$$
where $\mu_j(\cdot)=\mu(Q_j\cap \cdot)/\mu(Q_j)$ denotes the normalization of $\mu$ restricted on $Q_j$.

For an $f$-invariant measure $\mu$ and a finite partition $\mathcal P$, denote by 
$$\mathcal P^{n,f}=\bigvee_{j=0}^{n}f^{-j}(\mathcal P).$$
The metric entropy of $\mu$ associated to a partition $\mathcal P$ is defined as 
$$h_\mu(f,\mathcal P):=\lim_{n\to\infty}\frac{1}{n}H_\mu(\mathcal P^{n,f});$$
and the metric entropy of $\mu$ is defined as
$$h_\mu(f):=\sup \{h_\mu(f,\mathcal P): \mathcal P~\textrm{is a finite partition}\}.$$
Note that $h_{\mu}(f^q)=qh_{\mu}(f)$ and $h_\mu(f^q,\mathcal P)\leq q h_\mu(f,\mathcal P)$ for every finite partition $\mathcal{P}$ and every $q\in \NN$ (see \cite[Theorem 4.13]{Wal82}).

\begin{Remark}
	The $n$-th join of the partition $\mathcal P$ of $f$ is usually defined by $\cP^{n,f}=\bigvee_{j=0}^{n-1}f^{-j}(\mathcal P)$.
	However, since the Bowen metric $d_n$ defined later also involves the maximum over $0\leq i\leq n$, for convenience we define
	$\cP^{n,f}=\bigvee_{j=0}^{n}f^{-j}(\mathcal P)$.
	This modification does not change the definition of $h_{\mu}(f,\cP)$.
\end{Remark}

We finish this section by recalling the Lyapunov exponents.
Let $f:M\to M$ be a $\cC^1$ map. 
By the Oseledets theorem \cite{Ose68}, there exists an invariant set $\Gamma\subset M$ with total probability, 
i.e., $\mu(\Gamma)=1$ for any $f$-invariant measure $\mu$, such that for any $x\in\Gamma$, there are
\begin{itemize}
	\item  $\lambda_1(x,f)>\cdots>\lambda_{s(x)}(x,f)\geq -\infty$, which are $f$-invariant and measurable functions of $x$;
	\item a measurable flag $\{0\}\subseteq E_{s(x)}(x,f) \subseteq \cdots \subseteq E_1(x,f)=T_xM$ with $D_xfE_i(x,f)\subset E_i(f(x),f)$;
\end{itemize}
such that for any $v\in E_i(x,f)\setminus E_{i+1}(x,f)$ with $1\le i\le s(x)$ ($E_{s(x)+1}(x,f)=\{0\}$), one has
$$\lim_{n\to \infty}\frac{1}{n}\log\|D_xf^n(v)\|=\lambda_i(x,f).$$
These numbers $\{\lambda_i(x,f)\}_{i=1}^{s(x)}$ are called the \emph{Lyapunov exponents} of $ x\in \Gamma$.
Given $x\in\Gamma$, the sum of positive Lyapunov exponents is an important quantity used to describe the complexity of the dynamics. 
We define it as
$$\lambda_\Sigma^+(x,f):=\sum_{i=1}^{s(x)}\big[\dim E_i(x,f)-\dim E_{i+1}(x,f)\big] \cdot \max\big\{0,~\lambda_i(x,f)\big\}.$$
For an invariant measure $\mu$ of $f$, define the \emph{the sum of positive Lyapunov exponents} of $\mu$ by
$\lambda_\Sigma^{+}(\mu,f)=\int \lambda_\Sigma^+(x,f){\rm d}\mu(x)$.
Note that $\lambda^+_{\Sigma}(\cdot, f)$ is a non-negative affine upper semi-continuous function on the set of all $f$-invariant measures.

By the Oseledets theorem \cite{Ose68}, the sum of the positive Lyapunov exponents admits an equivalent definition in terms of exterior powers.
For $1\leq k\leq d$ and $x\in M$, let $\wedge^k D_xf$ denote the induced action of $D_xf$ on the $k$-th exterior power.
Define the sequence of functions $\{\Phi_n\}_{n\geq1}$ by
$$\Phi_{n}(x):=\max_{1\leq k\leq d} \log^{+} \|\wedge^k D_xf^n\|.$$
Then $\{\Phi_n\}_{n\geq1}$ is a sub-additive sequence of functions.
By the sub-additive ergodic theorem, for every $f$-invariant measure $\mu$, there exists a full $\mu$-measure set $X_{\mu}$ such that for every $x\in X_{\mu}$, the limit
$$\Phi^{\ast}(x):=\lim_{n\to \infty} \frac{1}{n} \max_{1\leq k\leq d} \log^{+} \|\wedge^k D_xf^n\|$$
exists.
Moreover, by the Oseledets theorem \cite{Ose68},  $\Phi^{\ast}(x)=\lambda^{+}_{\Sigma}(x,f)$ for $\mu$-almost every $x$.
Therefore, 
\begin{equation}\label{eq:positveLE}
	\lim_{n\to \infty}  \frac{1}{n} \int \Phi_{n}(x)~{\rm d}\mu(x)
	= \inf_{n>0} \frac{1}{n} \int \Phi_{n}(x)~{\rm d}\mu(x)
	= \int~\Phi^{\ast}(x)~{\rm d}\mu(x)=\lambda^+_{\Sigma}(\mu,f).
\end{equation}
Since every $\Phi_n$ is continuous, it follows that $\mu \mapsto \lambda^+_{\Sigma}(\mu,f)$ is upper semi-continuous.

\section{Preliminaries on linear algebra}\label{SEC:3}
Let $\big(V, (\cdot,\cdot)_{V}\big)$ and $\big(W, (\cdot,\cdot)_{W}\big)$ be finite-dimensional Euclidean inner product spaces, where $\dim V=k$ and $\dim W=s$.
Assume that $A:V\to W$ is a linear map.
The singular values of $A$ are defined as the positive square roots of the eigenvalues of the operator $A\cdot A^T:W\to W$ counted with multiplicity, where $A^T:W\to V$ is the adjoint linear map of $A$.
If $A$ is surjective (which requires $k\geq s$), then $A$ has exactly $s$-nonzero singular values, counted with multiplicity. 
We denote all the singular values of $A$ by
\begin{equation}\label{eq:dofPhi}
	\sigma_1(A)\geq \sigma_{2}(A) \geq \cdots \geq \sigma_{s}(A)\geq 0.
\end{equation}
Let $\log^+a=\max\{\log a,0\}$ and denote by $\wedge^t A$ the $t$-th exterior power of $A$ for $1\leq t\leq s$. 
Define
$$\Phi(A)=\max_{1\leq t \leq s} \log^{+}  \|\wedge^t A\|= \sum_{t=1}^{s} \log^{+} \sigma_{t}(A).$$

In general, the singular values of $A$ depend on the inner products on $V$ and $W$.  
However, they remain unchanged if both inner products are multiplied by the same positive constant.  
Indeed, for $a>0$, replacing the inner products by $(\cdot,\cdot)'_V:=a(\cdot,\cdot)_V$ and  $(\cdot,\cdot)'_W:=a(\cdot,\cdot)_W$ does not change the adjoint $A^T$ and the eigenvalues of $A\cdot A^T$.

Let $\big\{\big(E_k, (\cdot,\cdot)_k\big) \big\}_{k=0}^{n}$ be a family of $d$-dimensional Euclidean inner product linear spaces.
For each $0\leq k<n$, assume that $A_k:E_k\to E_{k+1}$ is a linear map.
Let 
$$V=\bigoplus_{k=0}^n E_k,~~~W=\bigoplus_{k=1}^n E_k.$$
For $v=(v_0,\cdots,v_n)\in V$, $v'=(v'_0,\cdots,v'_n)\in V$ and $w=(w_1,\cdots,w_n)\in W$, $w'=(w'_1,\cdots,w'_n)\in W$,
define the inner products 
$$(v,v')_{V} = \frac{1}{n+1}\sum_{k=0}^{n} (v_k,v'_k)_k,~~(w,w')_{W} = \frac{1}{n+1} \sum_{k=1}^{n} (w_k,w'_k)_k.$$
Let $\cL :V\to W$ denote the linear map 
$$\cL(v_0,\cdots,v_n)=(v_1-A_0v_0,~v_2-A_1v_1,~\cdots,v_n-A_{n-1}v_{n-1}).$$
Note that $\cL$ is surjective.
Indeed, for any $(u_1,\cdots,u_n)\in W$, fix an arbitrary $v_0\in E_0$ and define recursively $v_{k+1}=A_k(v_k)+u_{k+1}$ for each $0\leq k<n$, then $\cL(v_0,\cdots,v_n)=(u_1,\cdots,u_n)$.
Hence, $\cL$ has $nd$-nonzero singular values.
\begin{Lemma}\label{Lem:A1}
	Denote by $\sigma_{1}(\cL)\geq \cdots \geq \sigma_{nd}(\cL)$ all the nonzero singular values of $\cL$.
	Then, we have 
	$$\sum_{j=1}^{nd} \log^{+} [\sigma_{j}(\cL)]^{-1}\leq nd+\sum_{k=0}^{n-1}\Phi(A_k)-\Phi(A_{n-1}\circ A_{n-2} \circ \cdots \circ A_0).$$
\end{Lemma}
\begin{proof}
	Denote by $\cL^{T}:W\to V$ the adjoint of $\cL$.
	Then $\big\{[\sigma_{1}(\cL)]^2,\cdots,[\sigma_{nd}(\cL)]^2\big\}$ are precisely the eigenvalues of $\cL \circ \cL^T:W\to W$.
	Therefore
	\begin{align*}
		\sum_{j=1}^{nd} \log^{+} \sigma_{j}(\cL)^{-1}&=\frac{1}{2} \sum_{j=1}^{nd} \log^{+}  \sigma_{j}(\cL)^{-2}\leq \frac{1}{2} \sum_{j=1}^{nd} \log \big[1+ \sigma_{j}(\cL)^{-2}\big] \\
		&=\frac{1}{2} \log~\det\big(\Id|_{W}+ \cL \circ \cL^T \big)-\frac{1}{2} \log~\det\big(\cL \circ \cL^T \big).
	\end{align*}
	Note that 	
	\begin{align*}
		\cL \circ \cL^T =
		\begin{pmatrix}
			\Id|_{E_1}+A_0A_0^T&-A_1^T&\cdots&0\\
			-A_1&\Id|_{E_2}+A_1A_1^T&\ddots&\vdots\\
			\vdots&\ddots&\ddots&-A_{n-1}^T\\
			0&\cdots&-A_{n-1}&\Id|_{E_{n}}+A_{n-1}A_{n-1}^T
		\end{pmatrix}.
	\end{align*}
	Hence, by Hadamard-Fischer inequality we have
	\begin{align*}
		\frac{1}{2} \log~\det\big(\Id|_{W}+ \cL\cL^T \big)
		&\leq \frac{1}{2} \sum_{j=1}^{n}~\log~\det\big(2\Id|_{E_{j}}+ A_{j-1}A_{j-1}^T \big) \\
		&\leq  \sum_{j=1}^{n}~ \sum_{i=1}^{d}~\frac{1}{2}\log~\big(2+[\sigma_i(A_{j-1})]^2\big) \\
		&\leq  \sum_{j=1}^{n}~ \sum_{i=1}^{d}~1+\log^{+}~[\sigma_i(A_{j-1})]\\
		&=nd+\sum_{j=0}^{n-1}\Phi(A_j).
	\end{align*}
	
	Next, we estimate $\det\big(\cL \circ \cL^T \big)$. 
	Recall that 
	$$\det\big(\cL \circ \cL^T \big)= \max \big\{[\det (\cL|_{E})]^2:~E\subset V,~{\rm \dim}E=nd \big\}.$$
	Let $A=A_{n-1} \circ \cdots \circ A_0$.
	Choose $0\leq m\leq d$ such that $A$ has exactly $m$ singular values $> 1$. 
	Let $F_0^+\subset E_0$ be the right singular subspace of $A$ corresponding to the singular values strictly larger than $1$, and set $F_n^+:=A(F_0^+)\subset E_n$.
	Then, $$E:=F_0^{+}\oplus E_1\oplus \cdots \oplus E_{n-1} \oplus (F_n^+)^{\perp}\subset V$$
	is a $nd$-dimensional subspace of $V$, where $(F_n^+)^{\perp}$ denotes the orthogonal complement of $F_n$ in $\big(E_n,(\cdot,\cdot)_{n}\big)$. 
	Hence, we have $\det(\cL \circ \cL^T)\geq [\det(\cL|_{E})]^2$.
	It suffices to show that 
	\begin{equation}\label{eq:detcL}
		\log \det(\cL|_{E})=\Phi(A_{n-1}\circ A_{n-2} \circ \cdots \circ A_0).
	\end{equation}
	Define the linear map $T_1:E \to E$ and $T_2:W\to W$ by 
	\begin{align*}
		T_1(u,v_1,\cdots,v_{n-1},w)&=(u,~v_1-A_0(u),\cdots,~v_{n-1}-A_{n-2}(v_{n-2}),~w) \\
		T_2(u_1,u_2,\cdots,u_{n-1},u_n)&=\big(u_1,~u_2,~\cdots,~u_{n-1},~\sum_{j=1}^{n-1} A_{n-1}\circ \cdots \circ A_{j}(u_j)+u_n \big)
	\end{align*}
	A direct computation show that $\det(T_1)=1$ and $\det(T_2)=1$.
	Note that 
	$$T_2\circ \cL \circ T_1^{-1}(u,v_1,\cdots,v_{n-1},w)=\big(v_1,~v_2,~\cdots,~v_{n-1},~w-A(u)\big).$$
	Hence, 
	$$\det(\cL|_{E})=\det\big((u,w)\mapsto w-A(u)\big)=\det(A|_{F_0^+})=e^{\Phi(A)}.$$
	This completes the proof of the lemma.
\end{proof}

\begin{Remark}
	Note that both the definition of the singular values and the definition of $\Phi$ depend on the choice of inner products.
	Nevertheless, Lemma \ref{Lem:A1} remains valid if the inner products $(\cdot,\cdot)_V$ and $(\cdot,\cdot)_W$ are both replaced by $a(\cdot,\cdot)_V$ and $a(\cdot,\cdot)_W$ for some $a>0$.	
\end{Remark}

\begin{Lemma}\label{Lem:Acover}
	Let $a>0$ and $b>0$. 
	For every $d\in \NN$, every $d$-dimensional inner product linear space $\big(V,(\cdot,\cdot)_{V}\big)$ and every $x\in V$, there exists 
	$\{z_i\}_{i=1}^{N}\subset \{y:\|x-y\|_{V}\leq b \}$ such that 
	$$\big\{y:\|x-y\|_{V}\leq b \big\} \subset \bigcup_{i=1}^{N} \big\{z:\|z-z_i\|_{V}\leq  a\big\}~\text{and}~N\leq \left(1+\frac{2b}{a}\right)^d,$$
	where $\|\cdot\|_{V}$ is the norm induced by $(\cdot,\cdot)_{V}$.
\end{Lemma}
\begin{proof}
	Let ${\rm Vol}_{V}$ denote the Lebesgue measure on $V$ induced by the inner product $(\cdot,\cdot)_{V}$.
	For $x\in V$ and $r>0$, let $B(x,r):=\{y:\|x-y\|_{V}\leq r\}$.
	Then, for all $a,b>0$
	$${\rm Vol}_{V}(B(x,a))={\rm Vol}_{V}(B(0,a)),~~{\rm Vol}_{V}(B(0,a))=\big(\frac{a}{b}\big)^d \cdot {\rm Vol}_{V}(B(0,b)).$$ 
	Choose a maximal $a$-separated subset $\{z_1,\ldots,z_N\}\subset B(x,b)$: $\|z_{i}-z_{j}\|_{V}\geq a$ for every $i\neq j$ and $\{B(z_i,a)\}_{i=1}^{N}$ is a cover of $B(x,b)$.
	Since $\{B(z_i,a/2)\}_{i=1}^{N}$ have pairwise disjoint interiors and belong to $B(x,b+\frac{a}{2})$, it follows that 
	$$N\leq \frac{{\rm Vol}_{V}\big(B(x,b+\frac{a}{2})\big)}{\min_{z \in B(x,b)}{\rm Vol}_{V}(B(z,a/2))}\leq  \left(1+\frac{2b}{a}\right)^d.$$
	This completes the proof of the lemma.
\end{proof}

\section{A covering estimate for large Bowen balls}\label{SEC:4}
Suppose that $M$ is a compact Riemannian manifold without boundary and let $f:M\to M$ be a map.
For each $n>0$, we define $d_n^{f}(x,y)=\max_{0\leq k \leq n} d(f^k(x),f^k(y))$ and let 
\begin{equation}\label{eq:dofBowen}
	B_n(x,r,f):=\{y\in M:d_n^{f}(x,y)\leq r\}.
\end{equation}
The key step in proving the main theorem is to show that, for a fixed $\varepsilon$, the Bowen ball $B_n(x,\varepsilon,f)$ can be covered by how many Bowen balls $B_n(x,\delta,f)$, where $\delta$ can be chosen arbitrarily small compared with $\varepsilon$.
For each $x\in M$, we define
\begin{equation}\label{eq:DofD}
	\Delta^f_{n}(x)=\sum_{j=0}^{n-1} \max_{1\leq t\leq d} \log^{+} \|\wedge^t D_{f^j(x)}f\|- \max_{1\leq t\leq d}  \log^{+} \|\wedge^t D_{x}f^n\|\geq 0.
\end{equation}
and for each $Z\subset M$ define
$$r_n\big(Z,\delta,f\big):=\min  \Big\{N:~\exists\{Z_i\}_{i=1}^{N}~\text{s.t.}~\bigcup_{i=1}^{N} Z_{i}= Z~\text{and}~\diam_{d_n^f}(Z_i)\leq \delta \Big\}$$
The main result of this section is the following theorem.
\begin{Theorem}\label{Thm:key}
	Let $M$ be a compact Riemannian manifold of dimensional $d$.
	For $\alpha\in (0,1]$ and $\Upsilon>1$, there exist sequences of positive numbers $\{c_q\}_{q>0}$, $\{C_q(\alpha)\}_{q>0}$ and $\{\varepsilon_q(\alpha,\Upsilon)\}_{q>0}$, such that 
	\begin{enumerate}
	\item [(1)] $c_q\to 1$, $C_q(\alpha)/q\to 0$ as $q\to \infty$;
	\item [(2)] for every $\cC^{1,\alpha}$ map $f:~M\to M$ satisfying $\|f\|_{\mathcal C^{1,\alpha}}\leq \Upsilon$;
	\item [(3)] for every $0<\delta<\varepsilon\leq \varepsilon_q(\alpha,\Upsilon)$, there exists $D_q(d,\varepsilon,\delta)>0$ such that for every $n\in \NN$, $\Delta_n \in \NN$ and every $x\in M$
	$$\log r_n\Big(B_n(x,\varepsilon,f^q)\cap \big\{y\in M:\Delta^{f^q}_n(y)\in [\Delta_n,\Delta_n+1)\big\},~\delta,f^q\Big)\leq \frac{c_q}{\alpha}\cdot \Delta_n+nd C_q(\alpha)+D_q(d,\varepsilon,\delta).$$
	\end{enumerate}	 
\end{Theorem}
\begin{Remark}
	The three constants $c_q$, $C_q(\alpha)$, and $D_q(d,\varepsilon,\delta)$ all have explicit expressions, which are given in \eqref{eq:dofcq} and \eqref{eq:dofDq}.
	It is clear from these formulas that $c_q\to 1$, $C_q(\alpha)/q\to 0$ as $q\to \infty$, and that $D_q(d,\varepsilon,\delta)$ is independent of $x,n,\alpha$ and $f$.
\end{Remark}

In the following subsections, we prove this theorem. 
We first work in an abstract setting involving a sequence of maps between Euclidean inner product spaces.

\subsection{Mean Bowen metric}\label{Sec:3.1}
Let $\cE=\big\{\big(E_k, (\cdot,\cdot)_k\big) \big\}_{k=0}^{\infty}$ be a sequence of $d$-dimensional Euclidean inner product spaces. 
For each $k$, we denote by $\|\cdot\|_k$ the norm induced by $(\cdot,\cdot)_k$ on $E_k$.
Fix constants $0<r_0<1/4$ and $H>1$. 
We assume that 
$$\cF:=\big\{F_k:\{v\in E_k:~\|v\|_{k}\leq r_0\} \to E_{k+1}\big\}_{k=0}^{\infty}$$ 
is a sequence of maps with the following conditions:
\begin{itemize}
	\item  each $F_k$ is a $\cC^{1,\alpha}$ map, $F_k(0_k)=0_{k+1}$ and $\|F_k\|_{\cC^{1,\alpha}}\leq H$. 
\end{itemize}
For each $k$, by $\|F_k\|_{\cC^{1,\alpha}}\leq H$, we have
\begin{equation}\label{eq:C1alpha}
	\|F_k(v)-F_k(u)-D_uF_k(v-u)\|_{k+1} \leq H\cdot \|u-v\|^{1+\alpha}_{k},~\forall u,v\in \{t\in E_k:\|t\|_k\leq r_0\}. 
\end{equation}
Let $F^0=\Id|_{E_0}$ and let $F^n=F_{n-1}\circ\cdots\circ F_0$ for $n\geq1$.
Define the set 
\begin{equation}\label{eq:dofOmega}
	\Omega_n=\{v\in E_0: \|F^{k}(v)\|_{k} \leq r_0,~\forall 0\leq k\leq n\}.
\end{equation}
Then, $F^n$ is well-defined on $\Omega_n$. 
For each $x,y\in \Omega_n$, consider two metrics
\begin{equation}\label{eq:twometric}
	d_n(x,y)=\max_{0\leq j\leq n}\|F^j(x)-F^j(y)\|_{j},~~\overline{d}_n(x,y)=\Big(\frac{1}{n+1}\sum_{j=0}^{n}\|F^j(x)-F^j(y)\|_{j}^2\Big)^{1/2}.
\end{equation}
The reason we introduce the mean metric $\overline{d}$ here is that it is induced by an inner product.
For each $Z \subset \Omega_n$, define 
\begin{align*}
	r_n\big(Z,\delta,\cF\big):&=\min  \Big\{N:~\exists\{Z_i\}_{i=1}^{N}~\text{s.t.}~\bigcup_{i=1}^{N} Z_{i}= Z~\text{and}~\diam_{d_n}(Z_i)\leq \delta \Big\} \\
	\overline{r}_n\big(Z,\delta,\cF\big):&=\min  \Big\{N:~\exists\{Z_i\}_{i=1}^{N}~\text{s.t.}~\bigcup_{i=1}^{N} Z_{i}=Z~\text{and}~\diam_{\overline{d}_n}(Z_i)\leq \delta \Big\}.
\end{align*}
The following lemma gives a comparison between the mean metric and the Bowen metric.
\begin{Lemma}\label{Lem:key1}
	For every $0<\theta<1$ and every $0<r<r_0<1/4$ satisfying $r^{1-\theta}\leq 1/4$, and for every 
	$Z\subset \Omega_n$ with $\operatorname{diam}_{\overline{d}_n}(Z)\leq r$, we have
	$$\log r_n(Z,r^{\theta},\cF)\leq (n+1)\cdot \big( \HH(r^{1-\theta})+ dr^{1-\theta}\log (1+r^{-\theta})\big),$$
	where $\HH(t)=-t\log t-(1-t)\log (1-t)$.
\end{Lemma}
\begin{proof}
By Lemma \ref{Lem:Acover}, every ball $B(0_k,r_0)\subset E_k$ can be covered by at most $(1+4r_0r^{-\theta})^d$ balls of radius $r^\theta/2$, whose centers lie in $B(0_k,r_0)$.
Since $0<r_0<1/4$, there are at most $(1+r^{-\theta})^d$ such balls.
Fix a point $x\in Z$.
Then, for each $y\in Z$, by $r^{1-\theta}\leq 1/4$ we have
$$\#\big\{0\leq i\leq n: \|F^i(y)-F^i(x)\|>r^{\theta}/2\big\}\leq\frac{4(n+1)r^2}{r^{2\theta}}=4(n+1)r^{2(1-\theta)}\leq (n+1)r^{1-\theta}.$$
We first select the times at which the distance is large and then cover the corresponding points by balls of radius $r^\theta/2$. 
Therefore,
\begin{align*}
	\log r_n(Z,r^{\theta},\cF)
	&\leq \log~\sum_{j=0}^{\lfloor (n+1)r^{1-\theta}\rfloor} \binom{n+1}{j} \cdot (1+r^{-\theta})^{dj} \\
	&\leq (n+1)\cdot \big( \HH(r^{1-\theta})+ dr^{1-\theta}\log (1+r^{-\theta})\big),
\end{align*}
where, in the second inequality above, we have used the estimate
$$\sum_{j=0}^{n\kappa}\binom{n}{j}\le e^{n\mathbb H(\kappa)},~~~\forall n\in\mathbb N,~~\forall 0\le \kappa\leq 1/2.$$
This completes the proof of the lemma.
\end{proof}

\subsection{The key covering number estimate}\label{Sec:CofX}

Let $V_n=\bigoplus_{k=0}^n E_k$ and $W_n=\bigoplus_{k=1}^n E_k$.
For $v=(v_0,\cdots,v_n),~v'=(v'_0,\cdots,v'_n)\in V_n$ and $w=(w_1,\cdots,w_n),~w'=(w'_1,\cdots,w'_n)\in W_n$,
define the inner products 
$$(v,v')_{V}= \frac{1}{n+1}\sum_{k=0}^{n} (v_k,v'_k)_k,~~(w,w')_{W} = \frac{1}{n+1} \sum_{k=1}^{n} (w_k,w'_k)_k.$$
For each $x\in \Omega_n$, define a linear map $\cL_{x}:V_n\to W_n$ by
\begin{equation}\label{eq:cLx}
	\cL_{x}(v_0,\cdots,v_n):=\big(v_1-D_xF_0(v_0),~v_2-D_{F^1(x)}F_1(v_1),~\cdots,~v_n-D_{F^{n-1}(x)}F_{n-1}(v_{n-1})\big).
\end{equation} 
Note that $\cL_{x}$ is onto (see Section \ref{SEC:3}).
Hence, $\dim {\rm rank}(\cL_x)=nd$, $\dim \ker(\cL_x)=d$ and $\cL_x$ has $nd$-nonzero singular values.
Denote by $\sigma_1(x)\geq \sigma_2(x)\geq \cdots \geq \sigma_{nd}(x)>0$ all the non-zero singular values of $\cL_x$.
For each $a>0$, define
\begin{equation}\label{eq:dofmx}
	m_{x}^n(a):=\# \big\{1\leq j \leq nd:0<\sigma_{j}(x)\leq a \big\}.
\end{equation}

\begin{Lemma}\label{Lem:Key2}
	Fix constants $\lambda>0$ and $\theta\in (0,1)$. 
	Then, for every $0<r<r_{0}$ and every $Z\subset \Omega_n$ with $\diam_{\overline{d}_n}(Z)\leq r$ and $\diam_{d_n}(Z)\leq r^{\theta}$, we have
	$$\log \overline{r}_n\big(Z,re^{-\lambda},\cF\big)\leq  \inf_{x\in Z}\big(d+m_x^n(4He^{\lambda}r^{\alpha\theta})\big)\cdot \log(1+8e^\lambda).$$
\end{Lemma}
\begin{proof}
	Fix a point $x\in Z$. For each $z\in Z$, define
	$$v(z):=\big(z-x,~F^1(z)-F^1(x),~\cdots,~F^{n}(z)-F^n(x)\big)\in V_n.$$
	Then, we have $\|v(z)\|_{V}\leq r$.
	Write $\cL_x(v(z))=w(z)=(w(z)_1,\cdots,w(z)_n)$, where 
	$$ w(z)_i=(F_{i-1}-D_{F^{i-1}(x)}F_{i-1})(F^{i-1}(z))-(F_{i-1}-D_{F^{i-1}(x)}F_{i-1})(F^{i-1}(x)).$$
	By \eqref{eq:C1alpha}, we have
	$$ \|w(z)_i\|_{i}\leq  H\cdot \|F^{i-1}(z)-F^{i-1}(x)\|_{i-1}^{1+\alpha} \leq H \cdot r^{\theta\alpha}\cdot  \|F^{i-1}(z)-F^{i-1}(x)\|_{i-1}.$$
	Therefore, we have $\|\cL_x(v(z))\|_{W}\leq H \cdot r^{\theta\alpha} \|v(z)\|_{V}\leq H \cdot r^{\alpha\theta}\cdot r$.
	Denote by $V_x^{+}$ the right singular spaces of $\cL_x$ corresponding to singular values in $(0,~4He^{\lambda}r^{\alpha\theta}]$, and let 
	$$V_{x,n}^s:=\ker (\cL_x) \oplus V_x^{+}.$$
	Denote by $V_{x,n}^u:=(V_{x,n}^s)^{\perp}$ the orthogonal complement of $V_{x,n}^s$ in $V_n$.
	
	Write $v=v_s+v_u\in V_{x,n}^s\oplus V_{x,n}^u$ for each $v\in V_n$.
	If $v\in V_n$ satisfies $\|\cL_x(v)\|_{W}\leq H \cdot r^{\alpha\theta+1}$, then by $\cL_x(V_{x,n}^s) \perp  \cL_x(V_{x,n}^u)$, it follows that
	$$H \cdot r^{\alpha\theta+1}\geq\|\cL_x(v)\|_{W} \geq \|\cL_x(v_u)\|_{W}\geq 4He^{\lambda}r^{\alpha\theta}\|v_u\|_{V}~\Rightarrow~\|v_u\|_{V}\leq e^{-\lambda}r/4.$$
	By Lemma \ref{Lem:Acover}, there exists $z_1,\cdots,z_N\in V_{x,n}^s$ such that 
	$$\{v\in V_{x,n}^s: \|v\|_{V}\leq r\}\subset \bigcup_{i=1}^{N} \big\{v\in V_{x,n}^s: \|v-z_i\|_{V}\leq e^{-\lambda}r/4\big\},~~N\leq \Big(1+8e^{\lambda}\Big)^{\dim(V_{x,n}^s)}.$$
	For each $1\leq i\leq N$, we define
	$$K_{i}:=\big\{z\in Z:~\|v(z)_s-z_i\|_{V}\leq e^{-\lambda}r/4  \big\}.$$
	Since $\|v(z)_s\|_{V}\leq \|v(z)\|_{V} \leq r$ for each $z\in Z$, it follows that $\bigcup_{i=1}^{N}K_i=Z$.
	For each $z_1,z_2\in K_i$,
	\begin{align*}
	\overline{d}_n(z_1,z_2)=\|v(z_1)-v(z_2)\|_{V}
		&\leq \|v(z_1)_s-v(z_2)_s\|_{V}+\|v(z_1)_u-v(z_2)_u\|_{V} \\
		&\leq \frac{1}{2} e^{-\lambda}r+\frac{1}{2} e^{-\lambda}r\leq e^{-\lambda}r.
	\end{align*}
	Hence, $\diam_{\overline{d}_n}(K_i)\leq e^{-\lambda}r$.
	Note that $\dim(V_{x,n}^s)=d+ m_x^n(4He^{\lambda}r^{\alpha\theta})$.
	Therefore, we have 
	$$\log \overline{r}_n\big(Z,re^{-\lambda},\cF\big)\leq \big(d+m_x^n(4He^{\lambda}r^{\alpha\theta})\big)\cdot \log(1+8e^\lambda).$$
	Since $x\in Z$ was arbitrary, this completes the proof of the lemma.
\end{proof}

\begin{Proposition}\label{Prop:Cover-md}
	Fix constants $\lambda>0$ and $\theta\in (0,1)$.
	For every $r>0$ with $4r^{1-\theta}\leq 1$ and $4He^{\lambda}r^{\alpha \theta}\leq 1$, every $Z\subset \Omega_n$ with $\diam_{\overline{d}_n}(Z)\leq r$, and every $L>0$, there exist $N_L$ and $\{Z^L_i\}_{i=1}^{N_L}$ such that  
    $\bigcup_{i=1}^{N_L} Z_i^L=Z$, $\diam_{\overline{d}_n}(Z_i^L)\leq re^{-L\lambda}$ for every $1\leq i\leq N_L$, and
	\begin{align*}
		\log N_L \leq \log(1+8e^\lambda)\cdot &\Big(Ld+\sup_{z\in Z}~\sum_{s=0}^{L-1}~m_z^n\big(4He^{\lambda}(re^{-s\lambda})^{\alpha \theta}\big)  \Big)+ \\
		& \sum_{s=0}^{L-1}(n+1)\cdot \Big( \HH\big((re^{-s\lambda})^{1-\theta}\big)+ d\cdot (re^{-s\lambda})^{1-\theta}\log (1+ (re^{-s\lambda})^{-\theta})\Big).
	\end{align*}
\end{Proposition}
\begin{proof}
	First we consider the case $L=1$. 
	By Lemma \ref{Lem:key1}, there exist $X_{1}^1,X_{2}^1,\cdots,X_{N'_1}^1\subset Z$ such that 
	$\log N'_1\leq n\cdot \big( \HH(r^{1-\theta})+ dr^{1-\theta}\log (1+r^{-\theta})\big)$,
	$$\bigcup_{i=1}^{N'_1}X_{i}^1=Z,~~\diam_{\overline{d}_n}(X_i^1)\leq r,~~\diam_{d_n}(X_i^1)\leq r^{\theta},~~\forall 1\leq i \leq N'_1.$$
	For each $1\leq i \leq N'_1$, by Lemma \ref{Lem:Key2}, there exist $Z^1_{i,1},\cdots,Z^1_{i,N'_{1,i}}\subset X_i^1$ such that 
	$$\forall j,~\diam_{\overline{d}_n}(Z_{i,j}^1)\leq re^{-\lambda},~~\bigcup_{j=1}^{N'_{1,i}} Z_{i,j}^1=X_{i}^1,~~
	  \log N'_{1,i}\leq \inf_{x\in X_i^1}\big(d+m_x^n(4He^{\lambda}r^{\alpha\theta})\big)\cdot \log(1+8e^\lambda).$$
	Let $\{Z_1^1,\cdots Z^1_{N_{1}}\}=\bigcup_{i=1}^{N'_1}\{Z_{i,1}^1,\cdots,Z_{i,N'_{1,i}}^1\}$.
	Then, $\diam_{\overline{d}_n}(Z^1_i)\leq re^{-\lambda}$ and 
	\begin{align*}
		\log N_1
		&\leq (n+1)\cdot \big( \HH(r^{1-\theta})+ dr^{1-\theta}\log (1+r^{-\theta})\big)+\max_{i} \inf_{x\in X_i^1}\big(d+m_x^n(4He^{\lambda}r^{\alpha\theta})\big)\cdot \log(1+8e^\lambda).
	\end{align*}
	Using Lemma \ref{Lem:key1} for each $Z_i^1$, we obtain the families $\{X_{i,j}^2\}_{j=1}^{N'_{2,i}}$ such that
	$\bigcup_{j=1}^{N'_{2,i}}X_{i,j}^2=Z^1_i$, and for every $j$, $\diam_{d_n}(X_{i,j}^2)\leq (re^{-\lambda})^{\theta}$, while
	\begin{align*}
		\log N'_{2,i} \leq (n+1)\cdot \big( \HH((re^{-\lambda})^{1-\theta})+ d(re^{-\lambda})^{1-\theta}\log (1+(re^{-\lambda})^{-\theta})\big).
	\end{align*}
	Then, by Lemma \ref{Lem:Key2}, there exists a covering $Z^2_{i,j,1},\cdots,Z^2_{i,j,N'_{2,i,j}}\subset X_{i,j}^2$, such that, for every $t$, $\diam_{\overline{d}_n}(Z^2_{i,j,t})\leq re^{-2\lambda}$ and $\log N'_{2,i,j}\leq \inf_{x\in X_{i,j}^2} \big(d+m_x^n(4He^{\lambda}(re^{-s\lambda})^{\alpha\theta})\big)\cdot \log(1+8e^\lambda)$.
	Then, $\{Z^2_i\}_{i=1}^{N_2}$ is defined as the collection of all $\{Z_{i,j,t}^2:1\leq i\leq N_1,~1\leq j\leq N'_{2,i},~1\leq t\leq N'_{2,i,j}\}$. 
	Hence, 
	\begin{align*}
		\log N_2 \leq~\log(1+8e^\lambda)\cdot &\Big(2d+\max_{(i,j)}\inf_{z\in X_{i,j}^2}\sum_{s=0}^{1}~m_z^n\big(4He^{\lambda}(re^{-s\lambda})^{\alpha \theta}\big)  \Big)+ \\
		& \sum_{s=0}^{1}(n+1)\cdot \Big( \HH\big((re^{-s\lambda})^{1-\theta}\big)+ d\cdot (re^{-s\lambda})^{1-\theta}\log (1+ (re^{-s\lambda})^{-\theta})\Big)\\
		\leq~\log(1+8e^\lambda)\cdot &\Big(2d+\sup_{z\in Z}\sum_{s=0}^{1}~m_z^n\big(4He^{\lambda}(re^{-s\lambda})^{\alpha \theta}\big)  \Big)+ \\
		& \sum_{s=0}^{1}(n+1)\cdot \Big( \HH\big((re^{-s\lambda})^{1-\theta}\big)+ d\cdot (re^{-s\lambda})^{1-\theta}\log (1+ (re^{-s\lambda})^{-\theta})\Big).
	\end{align*}
	After repeating this procedure $L$ times, we obtain $N_L$ and a family $\{Z_i^L\}_{i=1}^{N_L}$ satisfying the conclusion of the proposition.
\end{proof}

In Proposition \ref{Prop:Cover-md}, we only obtain that the set is small with respect to the mean metric $\overline{d}_n$. 
However, applying Lemma \ref{Lem:key1} once again immediately yields the following conclusion:
\begin{Corollary}\label{Cor:Cover-dn}
	Fix two constants $\lambda>0$ and $\theta\in (0,1)$.
	For every $r>0$ with $4r^{1-\theta}\leq 1$ and $4He^{\lambda}r^{\alpha \theta}\leq 1$, every $Z\subset \Omega_n$ with $\diam_{\overline{d}_n}(Z)\leq r$, and every $L>0$, we have
	\begin{align*}
		\log r_n\big(Z,(re^{-L\lambda})^{\theta},\cF\big) \leq \log&(1+8e^\lambda)\cdot \Big(Ld+\sup_{z\in Z}~\sum_{s=0}^{L-1}~m_z^n\big(4He^{\lambda}(re^{-s\lambda})^{\alpha \theta}\big)  \Big)+ \\
		& (n+1)\cdot \sum_{s=0}^{L}\Big( \HH\big((re^{-s\lambda})^{1-\theta}\big)+ d\cdot (re^{-s\lambda})^{1-\theta}\log \big(1+ (re^{-s\lambda})^{-\theta}\big)\Big).
	\end{align*}
\end{Corollary}

\subsection{Estimate for the number of small singular values}
Next, we show that the complicated expression on the right-hand side of the inequality in Corollary \ref{Cor:Cover-dn} admits a good upper bound.
The key step is to estimate the number of small singular values.
Recall the definition of $m_x^n$ in \eqref{eq:dofmx}.
For each $x\in \Omega_n$, define 
\begin{equation}\label{eq:dofDF}
	\Delta_n^{\cF}(x):=\sum_{j=0}^{n-1}~\max_{1\leq k \leq d} \log^{+}~\|\wedge^k D_{F^j(x)}F_j \|-\max_{1\leq k \leq d}\log^{+} \|\wedge^k D_{x}F^n\|.
\end{equation}

\begin{Lemma}\label{Lem:sum1}
	For every $x\in \Omega_n$ and every $\chi>0$, we have
	$$\sum_{s=0}^{\infty}~m_x^n(e^{-s\chi})\leq (1+\chi^{-1})\cdot nd+\Delta_n^{\cF}(x)\cdot \chi^{-1}.$$
\end{Lemma}
\begin{proof}
	For each $s\in \NN$ and each $1\leq j\leq nd$, if $\sigma_{j}(x)\leq e^{-s\chi}\leq 1$, then $s\leq \chi^{-1} \log [\sigma_{j}(x)]^{-1}$.
	Hence,
	$$\#\big\{s\geq 0:~\sigma_{j}(x)\leq e^{-s\chi}\big\}\leq \frac{1}{\chi} \log^{+} \frac{1}{\sigma_{j}(x)}+1.$$
	Therefore,
	\begin{align*}
		\sum_{s=0}^{\infty}~m_x^n(e^{-s\chi})
		&\leq \sum_{j=1}^{nd} \# \big\{s\geq 0:~\sigma_{j}(x)\leq e^{-s\chi}\big\} \\
		&\leq \sum_{j=1}^{nd} \Big(\frac{1}{\chi} \log^{+} \frac{1}{\sigma_{j}(x)}+1\Big) \\
		&=nd+\frac{1}{\chi} \sum_{j=1}^{nd}\log^{+} \frac{1}{\sigma_{j}(x)}.
	\end{align*}
	By Lemma \ref{Lem:A1}, we have 
	$$ \sum_{j=1}^{nd}\log^{+} \frac{1}{\sigma_{j}(x)}\leq \Delta_n^{\cF}(x)+nd,$$
	which gives
	$$\sum_{s=0}^{\infty}~m_x^n(e^{-s\chi})\leq (1+\chi^{-1})\cdot nd+\Delta_n^{\cF}(x)\cdot \chi^{-1}.$$
	This completes the proof.
\end{proof}
Recall the function $\HH(t)=-t\log t-(1-t)\log (1-t)$ defined on $[0,1]$.
For convenience, we define the following function: for $s\in \NN$, $0<\theta<1$, $\lambda>0$, and $0<r<1/4$, let
$$\cH_{s}(\theta,\lambda,r)=\HH\big((re^{-s\lambda})^{1-\theta}\big)+ 
                            d\cdot (re^{-s\lambda})^{1-\theta}\log \big(1+ (re^{-s\lambda})^{-\theta}\big).$$
\begin{Lemma}\label{Lem:sum2}
	For $\lambda>1$ and $0<r<1/4$, we have 
	$$\sum_{s=0}^{\infty}~\cH_{s}(1-\lambda^{-1},\lambda,r)\leq 6d\lambda.$$
\end{Lemma}
\begin{proof}
	Let $\theta=1-\lambda^{-1}\in (0,1)$.
	Then, $(re^{-s\lambda})^{1-\theta}=r^{1/\lambda}e^{-s}$
	and $(re^{-s\lambda})^{-\theta}=\big(r^{1/\lambda}e^{-s}\big)^{-(\lambda-1)}$.
	Note that for every $0<t<1$,
	$$\HH(t)=-t\log t-(1-t)\log(1-t)\leq t\big(1-\log t\big).$$
	Moreover,
	\begin{align*}
		\log\left(1+t^{-(\lambda-1)}\right)
		&=-(\lambda-1)\log t+\log\big(1+t^{\lambda-1}\big)\\
		&\leq -(\lambda-1)\log t+t^{\lambda-1}.
	\end{align*}
	Applying these inequalities with $t=r^{1/\lambda}e^{-s}$,
	we obtain
	\begin{align*}
		\mathcal H_s\big(1-\lambda^{-1},\lambda,r\big) \leq
		r^{1/\lambda}e^{-s}+
		\bigl(1+d(\lambda-1)\bigr)r^{1/\lambda}e^{-s}\big(\log r^{-1/\lambda}+s\big)+
		dr e^{-s\lambda}.
	\end{align*}
	Summing over $s\geq0$
	we get
	\begin{align*}
		 & \sum_{s=0}^{\infty}\mathcal H_s \left(1-\lambda^{-1},\lambda,r\right) \\
    \leq~& \frac{r^{1/\lambda}}{1-e^{-1}} + \frac{ \bigl(1+d(\lambda-1)\bigr)r^{1/\lambda}\log(1/r)}{\lambda(1-e^{-1})} + \frac{\bigl(1+d(\lambda-1)\bigr)r^{1/\lambda}e^{-1}}{(1-e^{-1})^2} + \frac{dr}{1-e^{-\lambda}} \\
	\leq~& 2+d\lambda+2d\lambda+\frac{d}{2} \leq 6d\lambda.
	\end{align*}
	This completes the proof.
\end{proof}

\subsection{Proof of Theorem \ref{Thm:key}}
In this subsection, we complete the proof of Theorem \ref{Thm:key}. 
Let $M$ be a compact Riemannian manifold of dimensional $d$.
Fix constants $\alpha\in (0,1]$ and $\Upsilon>1$.
Let $f:M\to M$ be a $\cC^{1,\alpha}$ map satisfying $\|f\|_{\cC^{1,\alpha}}\leq \Upsilon$.

\begin{proof}[Proof of Theorem \ref{Thm:key}]

In order to apply the results established in the previous subsections, we first choose a uniform injectivity radius of $M$.
Since $M$ is  compact, one can choose $\rho_0>0$ and $\varepsilon_0>0$ such that for every $x\in M$
\begin{enumerate}
	\item[$ \bullet$]  $\exp_x: \{v\in T_xM: \|v\|\leq \rho_0 \} \rightarrow M$ is a $C^{\infty}$ embedding;
	\item[$ \bullet$]  $\exp_x(\{v\in T_xM: \|v\|\leq \rho_0 \})\supset B(x,\varepsilon_0)$;
	\item[$ \bullet$]  $\|D_{v}(\exp_x)\|\leq e^{1/2},~\forall \|v\|\leq \rho_0$ and $\|D_y(\exp_x^{-1})\|\leq e^{1/2},\, \forall y\in B(x,\varepsilon_0)$.
\end{enumerate}     
Fix $x\in M$ and $q\geq 2$. 
For each $n\in \NN$, define $\cE=\big\{ \big( E_n, (\cdot,\cdot)_n \big) \big\}_{n=0}^{\infty}$ and $\cF:=\big\{F_n\}_{n=0}^{\infty}$ by
\begin{enumerate}
	\item[(1)] $\big(E_n,(\cdot,\cdot)_n\big)=\big(T_{f^{qn}(x)}M,~(\cdot,\cdot)_{f^{qn}(x)}\big)$;
	\item[(2)] $r_0=\varepsilon_0 \cdot (2\Upsilon)^{-q}$;
	\item[(3)] $F_n=\exp_{f^{q(n+1)}(x)}^{-1} \circ f^q \circ\exp_{f^{qn}(x)}:~\{v\in T_{f^{qn}(x)}M:\|v\|\leq r_0\}\to T_{f^{q(n+1)}(x)}M$.
\end{enumerate}
By $\|f\|_{\cC^{1,\alpha}}\leq \Upsilon$ and $M$ is compact, for each fixed $q\in \NN$, there exists $H_{\alpha,M,q}(\Upsilon)>1$ such that for every $n\geq 0$
\begin{equation}\label{eq:dofH}
	H:=H_{\alpha,M,q}(\Upsilon)\geq \|F_n\|_{\cC^{1,\alpha}},~~\forall n\geq 0
\end{equation}
We now give the condition of $\varepsilon_q(\alpha,\Upsilon)$: set 
\begin{equation}\label{eq:dofepsiq}
    \varepsilon_q(\alpha,\Upsilon)=\frac{1}{4}\cdot \min\Big\{ \big[8He^{q/2}\big]^{-1/[\alpha(1-q^{-1/2})]},~~\varepsilon_0\cdot (8\Upsilon)^{-q},~~4^{-\sqrt{q}}\Big \},~~q\geq 2. 
\end{equation}
Recall the definition of $\Delta_{n}^{\cF}$ and $\Delta_{n}^{f^q}$ in \eqref{eq:dofDF} and \eqref{eq:DofD}. 
Then, by the choice of $r_0$ and the properties of the exponential map, for every $n>0$ and every $y \in B_n(x,r_0,f^q)$, we have
\begin{equation}\label{eq:roftwoD}
	|\Delta_{n}^{\cF}(\exp_x^{-1}{y})-\Delta_{n}^{f^q}(y)|\leq (n+1)d.
\end{equation}

Recall the definition of $\Omega_n$ in \eqref{eq:dofOmega} and two metrics $d_n,\overline{d}_n$ in \eqref{eq:twometric}.
For each $0<\varepsilon\leq \varepsilon_q(\alpha,\Upsilon)$ and each $\Delta_{n}\in \NN$, define
    $$Z:=\Big(\exp_{x}^{-1} \big(B_n(x,\varepsilon,f^q) \big)\Big) \cap \Big\{v\in \Omega_n:~\Delta_{n}^{\cF}(v)\leq \Delta_{n}+(n+1)d+1 \Big\}.$$
Then, we have 
    $$Z\subset \Omega_n,~~\text{and}~~\diam_{\overline{d}_n}(Z)\leq \diam_{d_n}(Z)\leq 4\varepsilon.$$
It follows from the properties of the exponential map and inequality \eqref{eq:roftwoD} that
\begin{equation}
	r_n\Big(B_n(x,\varepsilon,f^q)\cap \big\{x\in M:\Delta^{f^q}_n(x)\in [\Delta_n,\Delta_n+1)\big\},~\delta,f^q\Big)\leq r_n(Z,\delta/2,\cF).
\end{equation}
We now apply Corollary \ref{Cor:Cover-dn}. 
Recall we have fixed an integer $q\geq 2$.
Let $\lambda=q^{1/2}$, $\theta=1-q^{-1/2}$ and $r=4\varepsilon$, then by the choice of $\varepsilon_q(\alpha,\Upsilon)$ (see \eqref{eq:dofepsiq}) we have $4He^{\lambda}r^{\alpha \theta}<1$ and $r^{1-\theta}\leq 1/4$. 
For every $0<\delta<\varepsilon$, choose $L:=L_q(\varepsilon,\delta)\in \NN$ such that  
    $$L= \left\lfloor\frac{\log (8\varepsilon^{\theta}/\delta)}{\sqrt{q}-1} \right\rfloor+1,~~\Longrightarrow~~(e^{-L\lambda}\cdot r)^{\theta}<\delta/2.$$
Recall the definition of $m_z^n$ in \eqref{eq:dofmx} and $Z\subset \{v\in \Omega_n:~\Delta_{n}^{\cF}(v)\leq \Delta_{n}+(n+1)d+1\}$.
Then, by $4He^{\lambda}r^{\alpha \theta}<1$ and Lemma \ref{Lem:sum1} with the choice of parameters
    $$\chi=\alpha \lambda \theta= \alpha (\sqrt{q}-1),$$ 
we have
	$$\sup_{z\in Z}~\sum_{s=0}^{\infty}~m_z^n\big(4He^{\lambda}(re^{-s\lambda})^{\alpha \theta}\big)\leq \sup_{z\in Z}~\sum_{s=0}^{\infty}~m_z^n(e^{-s\chi})\leq nd+\frac{\Delta_n+4nd}{\alpha(\sqrt q-1)}.$$
Therefore, by Lemma \ref{Lem:sum2} and Corollary \ref{Cor:Cover-dn} we have
    $$\log r_n(Z,\delta/2,\cF)\leq nd\Big(20\sqrt{q}+\frac{4\log(1+8e^{\sqrt{q}})}{\alpha(\sqrt{q}-1)}\Big)+ \frac{\log(1+8e^{\sqrt{q}})}{\alpha(\sqrt{q}-1)} \cdot \Delta_{n} +\log(1+8e^{\sqrt{q}}) Ld.$$
Hence, for each $q\geq 2$ if we take 
\begin{equation}\label{eq:dofcq}
	c_q=\frac{\log(1+8e^{\sqrt{q}})}{\sqrt{q}-1},~~~C_q(\alpha)=20\sqrt{q}+\frac{4\log(1+8e^{\sqrt{q}})}{\alpha(\sqrt{q}-1)}
\end{equation}
and
\begin{equation}\label{eq:dofDq}
	D_q(d,\varepsilon,\delta)=d\cdot \log(1+8e^{\sqrt{q}})\cdot \left(\frac{\log \big(8\varepsilon^{1-q^{-1/2}}/\delta\big)}{\sqrt{q}-1}+1\right).
\end{equation}
Then, we have 
\begin{align*}
	\log r_n\Big(B_n(x,\varepsilon,f^q)\cap \big\{x\in M:\Delta^{f^q}_n(x)\in [\Delta_n,\Delta_n+1)\big\},~\delta,f^q\Big)
	&\leq \log r_n(Z,\delta/2,\cF)\\
	\leq&\frac{c_q}{\alpha}\cdot \Delta_n+nd\cdot C_q(\alpha)+D_q(d,\varepsilon,\delta).
\end{align*}
It follows immediately from \eqref{eq:dofcq} that $c_q\to1$ and $C_q(\alpha)/q\to0$ as $q\to \infty$.
Recall that we have assumed $q\geq 2$ throughout the above argument, and hence Theorem \ref{Thm:key} has been proved in this case.
For $q=1$, we fix $\lambda=3$ and $\theta=2/3$, and define $\varepsilon_1(\alpha,\Upsilon)$ according to \eqref{eq:dofepsiq} by taking ``$q=9$''. 
Repeating the above argument with this choice, we obtain the definitions of $\varepsilon_1(\alpha,\Upsilon)$, $c(1)$, $C_1(\alpha)$, and $D_1(d,\varepsilon,\delta)$.
This does not affect the limit as $q\to\infty$.
Therefore, we complete the proof of the theorem.
\end{proof}

\section{Proof of Theorem \ref{Thm:Main-Entropy}}\label{SEC:proof-Mth}
In this section, we complete the proof of Theorem \ref{Thm:Main-Entropy}. 
Once Theorem \ref{Thm:key} has been established, the remaining argument for the main theorem is standard.

\begin{proof}[Proof of Theorem \ref{Thm:Main-Entropy}]
	Fix constants $\alpha\in(0,1]$ and $\Upsilon>1$.
	Let $f:M\to M$ be a $\mathcal{C}^{1,\alpha}$ map of a $d$-dimensional compact manifold $M$ with $\|f\|_{\mathcal C^{1,\alpha}}\le\Upsilon$.
	For every $q>0$, Theorem \ref{Thm:key} provides constants $\varepsilon_q:=\varepsilon_q(\alpha,\Upsilon)$, $c_q$, and $C_q(\alpha)$ satisfying the required properties.
	Let $\mu$ be an $f$-invariant measure, and let $\cQ$ be a finite partition satisfying $\diam \cQ<\varepsilon_q$.
	
	Fix $\gamma>0$. 
	Then, there exists a finite partition $\cP$ with $\mu(\partial \cP)=0$ such that
	\begin{equation}\label{eq:Entropy1}
		h_{\mu}(f^q)\leq h_{\mu}(f^q,\cP)+\gamma.
	\end{equation} 
	Since $\mu(\partial \cP)=0$, there exists $0<\delta<\varepsilon_q$ such that $\mu(B(\partial \cP,2\delta))<(\gamma/ \log \#\cP)^2$.
	For each $n>0$, define
	$$K_n:=\left\{x\in M: \frac{1}{n+1}\#\Big\{0\leq j\leq n:f^{qj}(x)\in B(\partial \cP,2\delta)\Big\}\leq \frac{\gamma}{\log \#\cP}  \right\}.$$
	Then, $\mu(M\setminus K_n)\leq \gamma/(\log \#\cP)$.
	Recall the definition of $\Delta_n^{f^q}$ in \eqref{eq:DofD}. 
	Since $\|f\|_{\cC^{1,\alpha}}\leq \Upsilon$, it follows that $\Delta_n^{f^q}(x)\leq ndq \log^{+} \Upsilon$ for every $x\in M$.
	For each $0\leq k \leq ndq \log^{+} \Upsilon$, define
	$$\Delta_n(k):=\{x\in M:~k\leq \Delta_n^{f^q}(x)<k+1 \}.$$
	Let $\cE_n:=\{\Delta_n(k): 0\leq k \leq ndq \log^{+} \Upsilon\}$ be a finite partition of $M$, and let $\cR_n=\cE_n\vee \{K_n,M\setminus K_n\}$.
	Then, we have
	\begin{equation}\label{eq:Entropy2}
			h_{\mu}(f^q,\cP)=\lim_{n\to \infty} \frac{1}{n}H_{\mu}(\cP^{n,f^q})\leq \lim_{n\to \infty} \frac{1}{n}H_{\mu}(\cQ^{n,f^q})+\limsup_{n\to \infty} \frac{1}{n}H_{\mu}(\cP^{n,f^q}|\cQ^{n,f^q}).
	\end{equation}
	and, by $\# \cR_n\leq 2(ndq \log^{+} \Upsilon+1)$,
	\begin{align*}
		H_{\mu}(\cP^{n,f^q}|\cQ^{n,f^q})
		&\leq H_{\mu}\big(\cP^{n,f^q}|\cQ^{n,f^q}\vee \cR_n \big)+\log \# \cR_n \\ 
		&\leq H_{\mu}\big(\cP^{n,f^q}|\cQ^{n,f^q}\vee \cR_n \big)+\log~(2ndq \log^{+} \Upsilon+2).
	\end{align*}
	By $\mu(M\setminus K_n)\leq \gamma/(\log \#\cP)$ and $\log \#\cP^{n,f^q}\leq (n+1)\log \#\cP$, we have 
	\begin{align*}
		    ~~ &H_{\mu}(\cP^{n,f^q}|\cQ^{n,f^q}\vee \cR_n) \\
		\leq~~ &\sum_{E_n\in \cE_n\vee \cQ^{n,f^q}} \mu(E_n\cap K_n) \cdot  \log \#\{P_n\in \cP^{n,f^q}:P_n\cap E_n \cap K_n \neq \emptyset\}+\mu(M\setminus K_n) \log \#\cP^{n,f^q} \\
		\leq~~ &\sum_{E_n\in \cE_n\vee \cQ^{n,f^q}} \mu(E_n\cap K_n) \cdot  \log \#\{P_n\in \cP^{n,f^q}:P_n\cap E_n \cap K_n \neq \emptyset\}+(n+1)\gamma.
	\end{align*}
	By Theorem \ref{Thm:key}, for each $E_n=Q_n\cap \Delta_n(k)$ with $Q_n\in \cQ^{n,f^q}$ and $0\leq k \leq ndq \log^{+} \Upsilon$, there exist $Z^n_1,Z^n_2,\cdots,Z^n_{N_n}$ such that for every $1\leq i\leq N_n$
	$$ \forall 0 \leq j\leq n,~\diam(f^{qj}(Z^n_i))\leq \delta,~~\bigcup_{i=1}^{N_n} Z^n_i=E_n\cap K_n,~~\log N_n\leq \frac{c_q}{\alpha}\cdot k+nd\cdot C_q(\alpha)+D_q(d,\varepsilon,\delta).$$
	Therefore, for each $1\leq i\leq N_n$, we have
	$$ \# \{P_n\in \cP^{n,f^q}:~P_n \cap Z_i^n\neq \emptyset\}\leq \prod_{j=0}^{n} \#\{P\in \cP:~f^{qj}(Z_i^n)\cap P\neq \emptyset\}.$$
	Fix a point $z\in Z^n_i$, if $\#\{P\in \cP:~f^{qj}(Z_i^n)\cap P\neq \emptyset\}>1$, then by $\diam(f^{qj}(Z^n_i))\leq \delta$ we must have $f^{qj}(z)\in B(\partial\cP,2\delta)$. 
	Hence, by $Z^n_i\subset K_n$ we have
	$$ \# \{P_n\in \cP^{n,f^q}:P_n \cap Z_i^n\neq \emptyset\}\leq (\# \cP)^{\#\{0\leq i\leq n:~f^{qi}(z)\in B(\partial\cP,2\delta)\}}\leq e^{(n+1)\gamma}.$$
	Hence, for each $E_n=Q_n\cap \Delta_n(k)$ with $Q_n\in \cQ^{n,f^q}$ and $0\leq k \leq ndq \log^{+} \Upsilon$, we have
	$$ \log \# \{P_n\in \cP^{n,f^q}:P_n\cap E_n \cap K_n \neq \emptyset\} \leq \frac{c_q}{\alpha} \cdot k + nd \cdot C_q(\alpha) + D_q(d,\varepsilon,\delta) + (n+1)\gamma.$$
	Therefore,
	\begin{align*}
		      H_{\mu}(\cP^{n,f^q}|\cQ^{n,f^q}\vee \cR_n)
		&\leq \sum_{k}~\mu(\Delta_n(k)) \left[ \frac{c_q}{\alpha} \cdot k + nd \cdot C_q(\alpha) + D_q(d,\varepsilon,\delta) + (n+1)\gamma \right]+(n+1)\gamma \\
		&\leq \frac{c_q}{\alpha} \int \Delta_{n}^{f^q}(x)~{\rm d} \mu(x) + nd \cdot C_q(\alpha) + D_q(d,\varepsilon,\delta) + 2(n+1)\gamma.
	\end{align*}
	By letting $n\to\infty$ and using the definition of $\Delta_{n}^{f^q}(x)$ (see \eqref{eq:DofD}), together with the Birkhoff ergodic theorem and the Oseledets theorem \cite{Ose68} (see \eqref{eq:positveLE}), we obtain
	$$\limsup_{n\to \infty} \frac{1}{n}H_{\mu}(\cP^{n,f^q}|\cQ^{n,f^q})\leq \frac{c_q}{\alpha} \left(\int \max_{1\leq k\leq d}\log^{+}~ \|\wedge^k D_xf^q\|~{\rm d}\mu(x)-\lambda^{+}_{\Sigma}(\mu,f^q) \right)+ dC_q(\alpha)+2\gamma.$$
	Together with \eqref{eq:Entropy1} and \eqref{eq:Entropy2} we have
	$$h_{\mu}(f^q)\leq h_{\mu}(f^q,\cQ)+\frac{c_q}{\alpha} \left(\int \max_{1\leq k\leq d}\log^{+}~ \|\wedge^k D_xf^q\|~{\rm d}\mu(x)-\lambda^{+}_{\Sigma}(\mu,f^q)\right) + dC_q(\alpha)+3\gamma.$$
	Dividing both sides by $q$, by $h_{\mu}(f^q)=qh_{\mu}(f)$, $\lambda^{+}_{\Sigma}(\mu,f^q)=q\lambda^{+}_{\Sigma}(\mu,f)$, 
	$$h_{\mu}(f^q,\cQ)\leq q h_{\mu}(f,\cQ),~~\frac{1}{q}\int \max_{1\leq k\leq d}\log^{+}~ \|\wedge^k D_xf^q\|~{\rm d}\mu(x)-\lambda^{+}_{\Sigma}(\mu,f)\leq 2d \Upsilon$$
	and the fact that $\gamma>0$ is arbitrary, we obtain
	$$h_{\mu}(f)\leq h_{\mu}(f,\cQ)+\frac{1}{\alpha} \left(\frac{1}{q}\int \max_{1\leq k\leq d}\log^{+}~ \|\wedge^k D_xf^q\|~{\rm d}\mu(x)-\lambda^{+}_{\Sigma}(\mu,f)\right) + \frac{dC_q(\alpha)}{q}+\frac{|c_q-1|2d \Upsilon}{\alpha}.$$
	By Theorem \ref{Thm:key}, we have $C_q(\alpha)/q\to 0$ and $c_q\to 1$ as $q\to \infty$. Let 
	$$\delta_q:=\delta_{q}(\alpha,d,\Upsilon):= d\cdot \left(\frac{C_q(\alpha)}{q}+\frac{2|c_q-1|\Upsilon}{\alpha}\right).$$
	Then, $\delta_q\to 0$ as $q\to \infty$, and, for every $q\in \NN$, every $f$-invariant measure $\mu$ and every finite partition $\cQ$ with $\diam \cQ<\varepsilon_q$, we have
	$$h_{\mu}(f)\leq h_{\mu}(f,\cQ)+\frac{1}{\alpha} \left(\frac{1}{q}\int \max_{1\leq k\leq d}\log^{+}~ \|\wedge^k D_xf^q\|~{\rm d}\mu(x)-\lambda^{+}_{\Sigma}(\mu,f)\right) + \delta_q.$$
	This completes the proof.
\end{proof}

\section{Proof of Theorem \ref{Thm:Main-USC}}
We now turn to the proof of Theorem \ref{Thm:Main-USC}.
The proof of Theorem \ref{Thm:Main-USC} is essentially a direct application of Theorem \ref{Thm:Main-Entropy}.
\begin{proof}[Proof of Theorem \ref{Thm:Main-USC}]
	Under the assumptions of Theorem \ref{Thm:Main-USC}, let $\{f_n\}$ be a sequence of maps converging to $f$ in the $\cC^{1,\alpha}$ topology. 
	Choose $\Upsilon>1$ such that $\|f\|_{\cC^{1,\alpha}}<\Upsilon$ and $\|f_n\|_{\cC^{1,\alpha}}<\Upsilon$ for every $n\geq 0$.
	For convenience, for each $q\in \NN$ we denote 
	$$q(g,\nu):=\frac{1}{q}\int \max_{1\leq k\leq d}~\log^{+} ~ \|\wedge^k D_xg^q\|~{\rm d}\nu(x),\quad ~g\in \{f,~f_1,\cdots\},~\nu \in \{\mu,~\mu_1,\cdots\}.$$
	Note that $q(f,\mu) \to \lambda_{\Sigma}^{+}(\mu,f)$ as $q\to \infty$.
	Given $\varepsilon>0$, we first choose $q>0$ such that
	$$\delta_q(\alpha,d,\Upsilon)\leq \varepsilon~~\text{and}~~~q(f,\mu)-\lambda_{\Sigma}^{+}(\mu,f)<\alpha\varepsilon,$$
	where $\delta_q(\alpha,d,\Upsilon)$ is the constant given by Theorem \ref{Thm:Main-Entropy}.
	Let $\varepsilon_q>0$ be the finite constant given by Theorem \ref{Thm:Main-Entropy}, and let $\cQ$ be a finite partition whose boundary is piecewise smooth, such that $\diam \cQ<\varepsilon_q$ and $\mu(\partial \cQ)=0$.
	Then, by Theorem \ref{Thm:Main-Entropy}, for each $n>0$ we have 
	$$h_{\mu_n}(f_n)\leq h_{\mu_n}(f_n,\cQ)+\frac{1}{\alpha}\Big(\frac{1}{q}\int \max_{1\leq k\leq d}~\log^{+} ~ \|\wedge^k D_xf_n^q\|~{\rm d}\mu_n(x)-\lambda^{+}_{\Sigma}(\mu_n,f_n)\Big)+\varepsilon.$$
	Letting $n\to\infty$, we have $q(f_n,\mu_n)\to q(f,\mu)$. 
	Since $\mu(\partial\cQ)=0$, it follows that
	$$\limsup_{n\to\infty} h_{\mu_n}(f_n,\cQ)\leq h_{\mu}(f,\cQ).$$
	Therefore, by Theorem \ref{Thm:Main-Entropy} and  the choice of $q$
	\begin{align*}
		\limsup_{n \rightarrow\infty} h_{\mu_n}(f_n)-h_{\mu}(f)&\leq \limsup_{n \rightarrow\infty} h_{\mu_n}(f_n)-h_{\mu}(f,\cQ) \\
		&\leq \alpha^{-1}\Big(q(f,\mu)-\liminf_{n \to\infty} \lambda^{+}_{\Sigma}(\mu_n,f_n)\Big) +\varepsilon \\
	\big(\text{by}~q(f,\mu)\leq \lambda^{+}_{\Sigma}(\mu,f)+\alpha \varepsilon\big)~~~~~&\leq \alpha^{-1}\Big(\limsup_{n\to \infty} \big[\lambda^{+}_{\Sigma}(\mu,f)- \lambda^{+}_{\Sigma}(\mu_n,f_n)\big]\Big) +2\varepsilon.
	\end{align*}
	Since $\varepsilon>0$ was arbitrary, this completes the proof of Theorem \ref{Thm:Main-USC}.
\end{proof}

\appendix
\section{Entropy structure}\label{SEC:A}
In this section, we recall the definition of entropy structures and the statement of the symbolic extension theorem. 
The notations presented here is mainly based on \cite[Section 8]{Dow11} and \cite[Section 6]{Dow05}.
Throughout this section, we assume that $T:X\to X$ is a continuous map on a compact metric space $X$ with $h_{\rm top}(T)<\infty$.

We denote by $\PP_{T}(X)$ the set of all $T$-invariant measures.
Fix a decreasing sequence $\{\varepsilon_n\}_{n\geq1}$ converging to zero, define a sequence of functions $\{h_n(\cdot)\}_{n>0}$ on $\mathbb{P}_{T}(X)$ by
\begin{equation}\label{eq:dofES}
	h_n(\mu):=\inf\left\{ h_{\mu}(T,\mathcal Q): \diam(\cQ)<\varepsilon_n \right\},~~\forall \mu\in\mathbb{P}_{T}(X),
\end{equation}
where the infimum is taken over all finite Borel partitions $\mathcal Q$ of $X$.
The goal of this section is to show that $\mathcal H:=\{h_n\}$ is a entropy structure defined in \cite{Dow05} and \cite{Dow11}.
We start with the definition of entropy structures.

Let $\cH_1=\{h^1_n(\cdot)\}_{n>0}$ and $\cH_2=\{h^2_k(\cdot)\}_{k>0}$ be two increasing nets ($h^i_n(\cdot)\leq  h^i_{n+1}(\cdot)$ for $i=1,2$) of nonnegative functions defined on $\PP_{T}(X)$.
We call that $\cH_1$ is \textit{uniformly equivalent} to $\cH_2$, if for every $\gamma>0$ and every $k>0$, there exists $\ell:=\ell_{\gamma,k}>0$ such that $h^1_{\ell}(\mu)>h^2_{k}(\mu)-\gamma$ and $h^2_{\ell}(\mu)>h^1_{k}(\mu)-\gamma$ for every $\mu\in \PP_T(X)$.

An entropy structure is an equivalence class of increasing nets of functions on $\mathbb{P}_{T}(X)$ under the uniform equivalence relation above. 
In \cite[Section 8.3.1]{Dow11}, Downarowicz introduced three types of entropy structures: the \textit{Romagnoli entropy structure} $\cH^{\rm R}$, the \textit{continuous functions entropy structure} $\cH^{\rm F}$, and the \textit{Newhouse entropy structure} $\cH^{\rm N}$.
In \cite[Section 8.4]{Dow11}, he proved that these three entropy structures are uniformly equivalent. Hence, an entropy structure can be viewed as the equivalence class of uniformly equivalent increasing nets of nonnegative functions defined on $\mathbb{P}_{T}(X)$ containing these three increasing nets:
$\mathcal H^{\rm R}$, $\mathcal H^{\rm F}$, and $\mathcal H^{\rm N}$.
In other words, to show that the sequence of functions $\mathcal H:=\{h_n\}$ in \eqref{eq:dofES} forms an entropy structure, it suffices to prove that $\mathcal H$ is uniformly equivalent to $\mathcal H^{\rm F}$.

We next recall the definition of the continuous function entropy structure $\mathcal H^{\rm F}$. 
This definition can be found in \cite[Section 6.2]{Dow05} and \cite[Definition 8.3.11]{Dow11}.
Let $f:X \to [0,1]$ be a continuous function. 
There are two-element partition $\cP_f$ of $X\times [0,1]$
$$\cP_f:=\big\{ \{(x,t)\in X \times [0,1]:f(x)\geq t\},~\{(x,t)\in X \times [0,1]:f(x)<t\} \big\}.$$
For a finite family $\cF$ of such functions, we let $\cP^{\cF}:=\bigvee_{f\in \cF} \cP_{f}$.
Let $m$ denotes the Lebesgue measure on the interval $[0,1]$.
For each $\mu \in \PP_T(X)$, we define 
$$h_{\mu}(T,\cF):=h_{\mu \times m}(T\times \Id,\cP^{\cF}).$$
Assume that $\{\mathcal F_n\}_{n\geq1}$ is a sequence of finite families of continuous functions taking values in $[0,1]$, such that
$\mathcal F_n\subset \mathcal F_{n+1}$ and $\bigvee_{n=1}^{\infty}\mathcal P^{\mathcal F_n}$ is the partition into singletons of $X\times (0,1)$.
Then, the function $\{h_n^{F}\}_{n\geq 1}$ forms a continuous-function entropy structure $\mathcal H^{\rm F}$, where
$$h_n^{F}(\mu)=h_{\mu}(T,\cF_n).$$

\begin{Proposition}\label{Prop:AP}
	For every decreasing sequence $\{\varepsilon_n\}_{n\geq1}$ converging to zero, the increasing net of nonnegative functions
	$\mathcal H=\{h_n\}_{n\geq1}$ defined in \eqref{eq:dofES} is uniformly equivalent to the continuous function entropy structure $\mathcal H^F$.
\end{Proposition}
\begin{proof}
	We first construct such a family of continuous functions.
	Choose a sequence of subsets $\{A_n\}_{n\geq 1}$ of $X$ such that each $A_n$ is a finite set and $\bigcup_{x\in A_n}B(x,\varepsilon_n/8)=X$.
	By the Urysohn lemma, for each $x\in A_n$ there is a continuous function $f_x$ taking values in $[0,1]$ such that   
	$$f_x(y)=1,~\forall y\in M~\text{with}~d(x,y)\leq \frac{\varepsilon_n}{8}~~\text{and}~~f_x(y)=0,~\forall y\in M~\text{with}~d(x,y)> \frac{\varepsilon_n}{4}.$$
    Let $\{q_1,q_2,\ldots\}$ be an enumeration of the rational numbers in $[0,1]$, and define the function $g_i\equiv q_i$ on $X$.
	Define $\mathcal F_1:=\{f_x:x\in A_1\}\cup \{g_1\}$ and
	$\mathcal F_n:=\{g_n\} \cup \{f_x:x\in A_n\}\cup \mathcal F_{n-1}$ for $n\geq2$.
	It is clear that $\mathcal F_n\subset\mathcal F_{n+1}$ and that
	$\bigvee_{n=1}^{\infty}\mathcal P^{\mathcal F_n}$ is the partition into singletons of $X\times(0,1)$.
	We define the continuous function entropy structure $\cH^F$ as above.

	Define $\cP_f^t:=\big\{ \{x :f(x)\geq t\},~\{x:f(x)<t\} \big\}$ and $\cP^{n,t}=\bigvee_{f\in \cF_n} \cP_{f}^{t}$.
	Then $\mathcal P_f^t$ and $\mathcal P^{n,t}$ are finite partitions of $X$, where $t\in[0,1]$ and $f\in\mathcal F_n$.
	By the definition of $h_{\mu}(T,\cF_n)$, we have
	$$h^F_n(\mu)=h_{\mu}(T,\cF_n)=\int_{0}^{1} h_{\mu}(T,\cP^{n,t})~{\rm d}t,~~\forall \mu\in \PP_T(X).$$
	For each $t>0$, we claim that $\diam (\cP^{n,t})< \varepsilon_n$.	
	Indeed, for every $y \in P_n^t\in \cP^{n,t}$, there exists $x \in A_n$ such that $d(x,y)\leq \varepsilon_n/8$.
	Then, by the definition of $\cP^{n,t}$ we have $P_n^t\subset \{z:f_x(z)\geq t\}\subset \{z:d(x,z)\leq \varepsilon_n/4\}$.
	Therefore, for every $P_n^t\in \cP^{n,t}$ we have $\diam (P_n^t)\leq \varepsilon_n/2$.
	Hence, for every $\mu\in \PP_{T}(X)$ we have
	$$h^F_n(\mu)=\int_{0}^{1} h_{\mu}(T,\cP^{n,t})~{\rm d}t \geq \int_{0}^{1} \inf\left\{ h_{\mu}(T,\mathcal Q): \diam(\cQ)<\varepsilon_n \right\}~{\rm d}t=h_n(\mu).$$
	This establishes one side of the inequality. 
	We next to prove the other side.
	
	Fix $\gamma>0$ and an integer $n>0$. 
	Since $\cF_n$ is finite, there exists $\delta>0$ such that
	$$d(x,y)\leq \delta~~\Longrightarrow~~\#\cF_n\cdot |f(x)-f(y)|< \gamma/ (\log 2).$$
	Let $\cQ$ be a finite partition of $X$ such that $\diam(\cQ)<\delta$.
	Then, for every $\mu\in \PP_{T}(X)$ we have 
	\begin{align*}
		h^F_n(\mu)=\int_{0}^{1} h_{\mu}(T,\cP^{n,t})~{\rm d}t&\leq \int_{0}^{1} h_{\mu}(T,\cQ)+ H_{\mu}(\cP^{n,t} |\cQ)~{\rm d}t,\\
		&\leq \int_{0}^{1} h_{\mu}(T,\cQ)~{\rm d}t+ \sum_{Q\in \cQ} \mu(Q)\cdot \int_{0}^{1} H_{\mu_Q}(\cP^{n,t})~{\rm d}t,
	\end{align*}
	where $\mu_{Q}:=\mu(Q\cap \cdot)/\mu(Q)$ whenever $\mu(Q)>0$.
	For each $Q\in \cQ$ and $f\in \cF_n$, since $\diam(Q)<\delta$, there are $0\leq a_{f,Q}\leq b_{f,Q}\leq 1$ such that $a_{f,Q}\leq f(z)\leq b_{f,Q}$ for every $z\in Q$.
	Note that, whenever $t<a_{f,Q}$ or $t>b_{f,Q}$,
	we have $H_{\mu_{Q}}(\cP^{t}_{f})=0$. 
	Therefore, 
	\begin{align*}
		\sum_{Q\in \cQ} \mu(Q)\cdot \int_{0}^{1} H_{\mu_Q}(\cP^{n,t})~{\rm d}t&\leq \sum_{Q\in \cQ}\mu(Q)\cdot \sum_{f\in \cF_n}  \int_{a_{f,Q}}^{b_{f,Q}} H_{\mu_{Q}}(\cP^{t}_{f})~{\rm d}t \\
		&\leq \sum_{Q\in \cQ} \mu(Q)\cdot \sum_{f\in \cF_n}(b_{f,Q}-a_{f,Q})\log 2<\gamma
	\end{align*}
	Since $\cQ$ is an arbitrary finite partition satisfying $\diam \cQ<\delta$, for every $\mu\in \PP_{T}(X)$ we have 
	$$h^F_n(\mu)\leq \inf\left\{ h_{\mu}(T,\mathcal Q): \diam(\cQ)<\delta \right\}+\gamma.$$
	Choose $\ell:=\ell_{n,\gamma}>0$ such that $\varepsilon_{\ell}<\delta$ and $\ell\geq n$. 
    Then, we have $h_{\ell}>h_{n}^F-\gamma$, and by the other inequality proved above, $h_{\ell}^F>h_{n}-\gamma$.
	This completes the proof of the proposition.
    \end{proof}
	Finally, we recall the symbolic extension theorem in the form stated for continuous maps in \cite[Theorem 3]{Bur12}; see also \cite[Theorem 5.5]{BD04} and \cite[Chapter 9]{Dow11}.
	
	\begin{Theorem}[Symbolic extension theorem] \label{Thm:A.1}
		Let $T:X\to X$ be a continuous map with finite topological entropy.
		Assume that $\lambda(\cdot)$ is a non-negative affine upper semi-continuous function defined on $\PP_{T}(X)$, such that for every $\gamma>0$, every entropy structure $\{h_n\}$, and every $\mu\in \PP_T(X)$, there exist $k_{\mu}>0$ such that 
		$$\limsup_{\nu~\text{ergodic}~\to \mu} \Big(h_{\nu}(T)-h_{k}(\nu)+\lambda(\nu)\Big)\leq \lambda(\mu)+\gamma,~~\forall k\geq k_{\mu}.$$
		Then, there exists a symbolic extension $\pi:(\Sigma,\sigma)\to (X,T)$ such that 
		$$\max\big\{h_{\nu}(\sigma): \nu~\text{is}~\sigma\text{-invariant and}~\pi_{\ast}(\nu)=\mu \big\}=h_{\mu}(T)+\lambda(\mu).$$
		In particular, 
		$$\inf\big\{h_{\rm top}(\sigma):(\Sigma,\sigma)~\text{is a symbolic extension of}~(X,T)\big\}\leq h_{\rm top}(T)+\sup\big\{\lambda(\mu):\mu\in \PP_{T}(X)\big\}.$$
	\end{Theorem}
	By the uniform equivalence relation, it suffices to verify the assumptions of Theorem \ref{Thm:A.1} for one particular entropy structure.
	Therefore, Theorem \ref{Thm:1.1} is an immediate consequence of Theorem \ref{Thm:A.1} and Proposition \ref{Prop:AP}.

\section*{Acknowledgements}
The authors gratefully acknowledge the support of the Tianyuan Mathematical Center in Northeast China.

\vskip 5pt

\flushleft{\bf Chiyi Luo} \\
\small School of Mathematics and Statistics, Jiangxi Normal University, Nanchang,   330022, P. R. China\\
\textit{E-mail:} \texttt{luochiyi98@gmail.com}\\

\flushleft{\bf Dawei Yang} \\
\small School of Mathematical Sciences,  Soochow University, Suzhou, 215006, P.R. China\\
\textit{E-mail:} \texttt{yangdw@suda.edu.cn}\\
\end{document}